\documentclass[11pt]{article}
\usepackage[T1]{fontenc}
\usepackage{lmodern}
\usepackage{amsmath,amsthm,amssymb}
\usepackage[usenames,dvipsnames]{xcolor}
\usepackage{enumerate}
\usepackage{graphicx}
\usepackage{cite}
\usepackage{comment}
\usepackage{oands}
\usepackage{tikz}
\usepackage{changepage}
\usepackage{bbm}
\usepackage{mathtools}
\usepackage[margin=1in]{geometry}
\usepackage[pagewise,mathlines]{lineno}
\usepackage{appendix}
\usepackage{stmaryrd}
\usepackage{multicol}
\usepackage{microtype}
\usepackage[colorinlistoftodos]{todonotes}
\usepackage{dsfont}

\usepackage[pdftitle={Supercritical Liouville quantum gravity and CLE$_4$},
  pdfauthor={Morris Ang and Ewain Gwynne},
colorlinks=true,linkcolor=NavyBlue,urlcolor=RoyalBlue,citecolor=PineGreen,bookmarks=true,bookmarksopen=true,bookmarksopenlevel=2,unicode=true,linktocpage]{hyperref}

\theoremstyle{plain}
\newtheorem{thm}{Theorem}[section]

\newtheorem{lem}[thm]{Lemma}
\newtheorem{prop}[thm]{Proposition}
\newtheorem{propq}[thm]{``Proposition''}
\newtheorem{conj}[thm]{Conjecture}

\def\@rst #1 #2other{#1}
\newcommand\MR[1]{\relax\ifhmode\unskip\spacefactor3000 \space\fi
  \MRhref{\expandafter\@rst #1 other}{#1}}
\newcommand{\MRhref}[2]{\href{http://www.ams.org/mathscinet-getitem?mr=#1}{MR#2}}

\theoremstyle{definition}
\newtheorem{defn}[thm]{Definition}
\newtheorem{remark}[thm]{Remark}

\newtheorem{prob}[thm]{Problem}

\numberwithin{equation}{section}

\newcommand{\dsb}{\begin{adjustwidth}{2.5em}{0pt}
\begin{footnotesize}}
\newcommand{\dse}{\end{footnotesize}
\end{adjustwidth}}

\newcommand{\ssb}{\begin{adjustwidth}{2.5em}{0pt}}
\newcommand{\sse}{\end{adjustwidth}}

\newcommand{\aryb}{\begin{eqnarray*}}
\newcommand{\arye}{\end{eqnarray*}}
\def\alb#1\ale{\begin{align*}#1\end{align*}}
\def\allb#1\alle{\begin{align}#1\end{align}}
\newcommand{\eqb}{\begin{equation}}
\newcommand{\eqe}{\end{equation}}
\newcommand{\eqbn}{\begin{equation*}}
\newcommand{\eqen}{\end{equation*}}

\newcommand{\BB}{\mathbbm}
\newcommand{\ol}{\overline}

\newcommand{\op}{\operatorname}

\newcommand{\eqD}{\overset{d}{=}}
\newcommand{\ep}{\varepsilon}

\newcommand{\wt}{\widetilde}
\newcommand{\wh}{\widehat}
\newcommand{\mcl}{\mathcal}

\newcommand{\bdy}{\partial}

\newcommand{\cc}{{\mathbf{c}}}
\newcommand{\ccL}{{\mathbf{c}_{\mathrm L}}}
\newcommand{\ccM}{{\mathbf{c}_{\mathrm M}}}

\let\originalleft\left
\let\originalright\right
\renewcommand{\left}{\mathopen{}\mathclose\bgroup\originalleft}
\renewcommand{\right}{\aftergroup\egroup\originalright}

\title{Supercritical Liouville quantum gravity and CLE$_4$}
 \date{ }
 \author{
\begin{tabular}{c} Morris Ang\\[-5pt]\small UC San Diego \end{tabular}
\begin{tabular}{c} Ewain Gwynne\\[-5pt]\small University of Chicago \end{tabular}  
}

\begin{document}

\maketitle

\begin{abstract}
We establish the first relationship between Schramm-Loewner evolution (SLE) and Liouville quantum gravity (LQG) in the supercritical (a.k.a.\ strongly coupled) phase, which corresponds to central charge values $\mathbf c_{\mathrm L} \in (1,25)$ or equivalently to complex values of $\gamma$ with $|\gamma|=2$.
More precisely, we introduce a canonical supercritical LQG surface with the topology of the disk.
We then show that for each $\mathbf c_{\mathrm L} \in (1,25)$ there is a coupling of this LQG surface with a conformal loop ensemble with parameter $\kappa=4$ (CLE$_4$) wherein the LQG surfaces parametrized by the regions enclosed by the CLE$_4$ loops are conditionally independent supercritical LQG disks given their boundary lengths. 
In this coupling, the CLE$_4$ is neither determined by nor independent from the LQG.
Guided by our coupling result, we exhibit a combinatorially natural family of loop-decorated random planar maps whose scaling limit we conjecture to be the supercritical LQG disk coupled to CLE$_4$. 
We include a substantial list of open problems. 
\end{abstract}

 

\tableofcontents

\bigskip
\noindent\textbf{Acknowledgments.}
We thank two anonymous referees for helpful comments on an earlier version of this article.
We also thank Nina Holden and Minjae Park for helpful comments on an even earlier version of this article. 
This work has benefited from enlightening discussions with many people, including Bruno Balthazar, Manan Bhatia, Ahmed Bou-Rabee, J\'er\'emie Bouttier, William Da Silva, Jian Ding, Yuyang Feng, Nina Holden, Jiaqi Liu, Minjae Park, Josh Pfeffer, Guillaume Remy, Scott Sheffield, Xin Sun, and Jinwoo Sung. 
M.A.\ was partially supported by the Simons Foundation as a Junior Fellow at the Simons Society of Fellows and by NSF grant DMS-2348201.
E.G.\ was partially supported by a Clay research fellowship and by NSF grant DMS-2245832. 
\medskip

\section{Introduction}
\label{sec-intro}

Liouville quantum gravity (LQG) is a canonical one-parameter family of models of random geometry in two dimensions. LQG was first studied in the physics literature in the 1980s~\cite{polyakov-qg1,david-conformal-gauge,dk-qg} in the setting of string theory and two-dimensional gravity (see Remark~\ref{remark-string}), and has been an extremely active area of research in mathematics for the past fifteen years. 
Some common choices of the parameter for LQG are the \textbf{(Liouville) central charge}\footnote{Some works on LQG also consider the \textbf{matter central charge}, which satisfies $\ccM =26-\ccL = 25-6Q^2$.}
 $\ccL > 1$, the \textbf{background charge} $Q > 0$, or the \textbf{coupling constant} $\gamma \in (0,2] \cup \{z\in\BB C : |z| =2\}$, which are related by the formulas
\eqb \label{eqn-lqg-parameters}
\ccL = 1+6Q^2 ,\quad Q = \frac{2}{\gamma}  +\frac{\gamma}{2} .
\eqe

One can consider LQG surfaces with the topology of any orientable surface.
In this paper we will consider LQG surfaces with the topology of the unit disk $\BB D$. 
Heuristically speaking, an LQG surface with central charge $\ccL$ is described by a random Riemannian metric tensor $g$ on $\BB D$ which is sampled from the ``uniform measure on Riemannian metric tensors, weighted by $(\det\Delta_g)^{-(26-\ccL)/2}$'', where $\Delta_g$ is the Laplace-Beltrami operator. This definition does not make literal sense, but LQG surfaces can be defined rigorously in various ways, as we will discuss below. 
 
\begin{defn} \label{def-phases}
We refer to the case when $\ccL > 25$ (equivalently $Q> 2$ or $\gamma\in (0,2)$) as the \textbf{subcritical} or \textbf{weakly coupled} phase. 
We refer to the case when $\ccL = 25$ ($Q=\gamma=2$) as the \textbf{critical} case.
We refer to the case when $\ccL \in (1,25)$ ($Q\in (0,2)$ or $\gamma \in \BB C$ with $|\gamma|=2$) as the \textbf{supercritical} or \textbf{strongly coupled} phase.
\end{defn}

Most works on LQG consider only the subcritical and critical cases. In these cases, the DDK ansatz~\cite{david-conformal-gauge,dk-qg} implies that the Riemannian metric tensor associated with LQG takes the form
\eqb \label{eqn-lqg-metric-tensor}
g = e^{\gamma \Phi} \,(dx^2 + dy^2) 
\eqe 
where $\Phi$ is a random generalized function on $\BB D$ which locally looks like the Gaussian free field (GFF).
We assume that the reader is familiar with the GFF; the unfamiliar reader can consult, e.g.,~\cite{shef-gff,pw-gff-notes,bp-lqg-notes}. See, e.g.,~\cite{gkr-bootstrap-survey} for an overview of how to make rigorous sense of the ``weighting by $(\det\Delta_g)^{-(26-\ccL)/2}$'' perspective using Polyakov's path integral formulation.

The metric tensor~\eqref{eqn-lqg-metric-tensor} still does not make rigorous sense, but one can nevertheless define various objects associated with it by approximating $\Phi$ by continuous functions. 
Moreover, following~\cite{shef-kpz,shef-zipper,wedges}, one can define LQG surfaces rigorously as domains decorated by random generalized functions, viewed modulo a conformal change of coordinates rule (see Definition~\ref{def-lqg-surface} below).
There is a vast literature on critical and subcritical LQG and its connections to various other mathematical objects, including SLE, random planar maps, conformal field theory, and random permutations. See~\cite{bp-lqg-notes,gwynne-ams-survey,sheffield-icm} for some introductory expository articles on critical and subcritical LQG.

In this paper, we will be primarily interested in the supercritical phase $\ccL \in (1,25)$.  
In this phase, $\gamma$ is complex so the metric tensor~\eqref{eqn-lqg-metric-tensor} is not defined even heuristically. 
In part for this reason, the supercritical phase is much more mysterious than the subcritical phase, even at a physics level of rigor.  
Nevertheless, supercritical LQG is potentially more interesting than subcritical LQG from the perspective of string theory and Yang-Mills theory, as we discuss in Remark~\ref{remark-string} below. 
Moreover, a number of works have studied supercritical LQG or Liouville CFT with $\ccL \in (1,25)$ from a non-probabilistic perspective, see, e.g.,~\cite{ambjorn-remarks,bh-c-ge1-matrix,david-c>1-barrier,fkv-c>1-I,fk-c>1-II,gervais-weak-to-strong,gn-locality-string-models,bg-new-critical-dim,suzuki-note,verlinde-quantization,seiberg-notes,ribault-cft}.
See~\cite[Section 2]{ghpr-central-charge} for additional discussion and references on supercritical LQG, including some discussion on the meaning of ``weighting by $(\det\Delta_g)^{-(26-\ccL)/2}$'' in this context. We emphasize that the supercritical phase of LQG as discussed in this paper has no direct connection to the purely atomic supercritical Gaussian multiplicative chaos measures for $\gamma > 2$, as considered, e.g., in~\cite{dup-dual-lqg,bjrv-gmt-duality,rhodes-vargas-review,wedges,bh-supercritical-gmc}: our notion of supercritical LQG corresponds to $\gamma\in\BB C$ with $|\gamma| = 2$, not $\gamma > 2$. 

The definition of LQG surfaces as domains decorated by generalized functions, viewed modulo conformal coordinate change, still makes sense when $\ccL \in (1,25)$ (see Definition~\ref{def-lqg-surface}).
Prior to this paper, the main achievement in the probabilistic study of supercritical LQG was the construction of the supercritical LQG metric (distance function)~\cite{dg-supercritical-lfpp,dg-uniqueness,pfeffer-supercritical-lqg} for all $\ccL\in (1,25)$, which extends the subcritical LQG metric constructed in~\cite{dddf-lfpp,gm-uniqueness} (see~\cite{ddg-metric-survey} for a survey of results related to the LQG metric).
Heuristically speaking, the reason why this is possible, even though $\gamma$ is complex, is that when defining the LQG distance function one should integrate $e^{\xi \Phi}$ along paths, where $\xi = \xi(\ccL) > 0$ is a number which is not known explicitly (when $\ccL > 25$, one has $\xi = \gamma/d_\gamma$, where $d_\gamma$ is the Hausdorff dimension of the LQG metric space). The parameter $\xi$ is real even when $\ccL \in (1,25)$. See~\cite{dg-uniqueness} for details.

The supercritical LQG metric is a random metric on a domain $U\subset\BB C$, which is given by a measurable function of a GFF-like random generalized function on $U$.
Unlike in the subcritical and critical~\cite{dg-critical-lqg} cases, the supercritical LQG metric does not induce the Euclidean topology on $U$. Rather, there is an uncountable, but zero Lebesgue measure, set of \textbf{singular points} in $U$ which lie at infinite LQG distance from every other point in $U$. See~\cite{dg-uniqueness} for details. Roughly speaking, singular points correspond to $\alpha$-thick points of the field $\Phi$ for $\alpha > Q$~\cite{pfeffer-supercritical-lqg}. The existence of singular points is related to the presence of tachyonic operators in Liouville conformal field theory for $\ccL \in (1,25)$, which give rise to infinite-diameter ``holes'' in the surface~\cite[Section 5.5]{seiberg-notes}.

In light of the construction of the supercritical LQG metric, it is natural to ask what other features of subcritical LQG can be  extended to the supercritical case. For example: 
\begin{enumerate}
\item Connections between LQG and Schramm-Loewner evolution (SLE), such as the \textbf{quantum zipper}~\cite{shef-zipper} and \textbf{mating of trees}~\cite{wedges}  (see~\cite{ghs-mating-survey} for a survey). Roughly speaking, these connections take the following form. Suppose we have a certain type LQG surface of central charge $\ccL \geq 25$ and a certain type of SLE$_\kappa$ curve, sampled independently from each other, whose parameters satisfy 
\eqb \label{eqn-matched}
\kappa \in \left\{\gamma^2 ,\frac{16}{\gamma^2} \right\} , \quad\text{equivalently} \quad   26-\ccL = \cc(\kappa) := 1 - 6\left(\frac{2}{\sqrt\kappa} - \frac{\sqrt\kappa}{2}\right)^2  .
\eqe
Then the LQG surfaces parametrized by the complementary connected components of the SLE$_\kappa$ curve are conditionally independent given the LQG lengths of their boundaries, and their laws can be described explicitly. \label{item-sle}
\item The probabilistic formulation of LQG in terms of conformal field theory~\cite{dkrv-lqg-sphere,hrv-disk,grv-higher-genus} and various exact formulas which come from this formulation (e.g.,~\cite{krv-dozz,gkrv-bootstrap,arsz-structure-constants}). \label{item-cft} 
\item Connections between LQG and random planar maps, as surveyed, e.g., in~\cite{ghs-mating-survey}. \label{item-rpm}
\item The LQG area and length measures, which heuristically speaking are the Riemannian length and area measures associated with~\eqref{eqn-lqg-metric-tensor}~\cite{rhodes-vargas-log-kpz,shef-kpz} and which can be constructed via Gaussian multiplicative chaos~\cite{kahane,rhodes-vargas-review,berestycki-gmt-elementary}. \label{item-measure}  
\end{enumerate}
Prior to this work, very little was known about analogs of any of these results in the supercritical case. 
See, e.g., the lists of open problems in~\cite{ghpr-central-charge,apps-central-charge} for further discussion.

\subsection{Main result: coupling of supercritical LQG and CLE$_4$}

The main result of this paper is an extension of the connection between LQG and SLE (Item~\ref{item-sle}) to the supercritical case. In particular, for each $\ccL \in (1,25)$ and each $L > 0$ we introduce the supercritical LQG disk with central charge $\ccL$ and boundary length $L$: this is a canonical central charge-$\ccL$ LQG surface described by a GFF-like random generalized function $\Phi$ on $\BB D$ (Definition~\ref{def-supercritical-disk}). See Definition~\ref{def-lqg-surface} for our definition of a central charge-$\ccL$ LQG surface.
We then prove, roughly speaking, the following (see Theorem~\ref{thm-cle4-coupling} for a precise statement).

\begin{thm}[Main result, informal statement] \label{thm-main-intro}
Let $\ccL \in (1,25)$ and let $\Phi$ be the field corresponding to a unit boundary length LQG disk of central charge $\ccL$. 
There is a coupling of $\Phi$ with a (non-nested) conformal loop ensemble (CLE) with parameter $\kappa=4$ in $\BB D$ with the following property. The central charge-$\ccL$ LQG surfaces obtained by restricting $\Phi$ to the regions enclosed by the CLE$_4$ loops (see Definition~\ref{def-lqg-surface}) are conditionally independent supercritical LQG disks given their boundary lengths. In this coupling, the CLE$_4$ is neither independent from nor determined by $\Phi$. 
\end{thm}

\begin{figure}[ht!]
\begin{center} 
\includegraphics[width=0.5\textwidth]{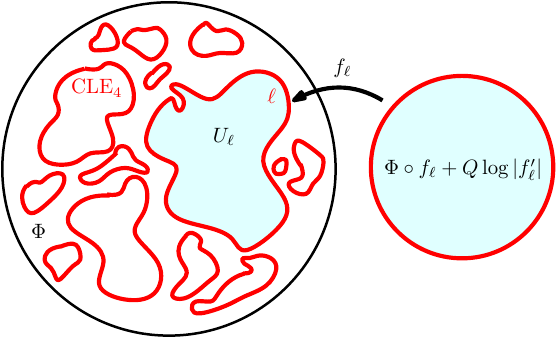}  
\caption{\label{fig-understanding-improved} 
Illustration of the statement of Theorem~\ref{thm-main-intro}. The precise version of the theorem (Theorem~\ref{thm-cle4-coupling}) says the following. Consider LQG surfaces obtained by restricting the field $\Phi$ to the regions $U_\ell$ enclosed by the CLE$_4$ loops, then viewing the pairs $(U_\ell,\Phi|_{U_\ell})$ modulo $Q$-LQG coordinate change, where $Q \in (0,2)$ is as in~\eqref{eqn-lqg-parameters} (see Definition~\ref{def-lqg-surface}). Then these LQG surfaces are conditionally independent supercritical LQG disks given their boundary lengths. 
}
\end{center}
\end{figure}

See Figure~\ref{fig-understanding-improved} for an illustration of the theorem statement. 
CLE$_\kappa$ for $\kappa\in (8/3,4]$ is a countable collection of disjoint loops which locally look like SLE$_\kappa$ curves, first introduced in~\cite{shef-cle}. 
A similar statement to Theorem~\ref{thm-main-intro}, for $\ccL \in (25,26)$, equivalently $\gamma \in (\sqrt{8/3},2)$, is proven as~\cite[Theorem 1.1]{msw-simple-cle-lqg}, and is extended to the critical case $\ccL=25$ in~\cite{ag-critical-cle4}. We re-state this result as Theorem~\ref{thm-subcritical-coupling} below. In this case, instead of a CLE$_4$ coupled to the LQG disk, one has a CLE$_{\kappa=\gamma^2}$ sampled independently from the LQG disk.

We will actually prove a stronger statement than Theorem~\ref{thm-main-intro}, which gives a coupling of the supercritical LQG disk with a whole infinitely-nested CLE$_4$ satisfying a certain Markovian property, see Definition~\ref{def-disk-cle4} and Theorem~\ref{thm-cle4-coupling-nested}. 

This coupling is explicit: the CLE$_4$ is viewed as the set of level lines or level loops of a GFF $\Psi$ (as in~\cite{shef-miller-cle4,asw-local-sets}, see Section~\ref{sec-nested}) and $\Phi$ is given by a linear combination of $\Psi$ and the field associated with a critical ($\ccL=25$) LQG disk. This coupling has a number of interesting probabilistic and geometric features, see Sections~\ref{sec-cle4-coupling} and~\ref{sec-comments} for details. Perhaps surprisingly, the proof of the main part of the coupling theorem (given in Section~\ref{sec-cle4-coupling}) only takes a little over a page, once one has the necessary setup. 

We expect that our techniques can also lead to other couplings of supercritical LQG surfaces with SLE$_4$-type objects, e.g., single SLE$_4$-type curves or the special local sets of the GFF considered, e.g., in~\cite{aru-sepulveda-2valued,als-isomorphism} (see Section~\ref{sec-open-problems}). One type of coupling of SLE and LQG which we do not know how to generalize to the supercritical case is the mating of trees theorem~\cite{wedges}. The reason for this is that we do not have a locally finite area measure associated with supercritical LQG, so we cannot parametrize a space-filling curve by supercritical LQG area.

A key idea in this paper is to take orthogonal linear combinations of independent fields (generalized functions) to get a new collection of fields with the same total central charge (see Appendix~\ref{sec-disk-rotate} for a general statement). This idea also plays a central role in the paper~\cite{ag-mismatched}. We expect that this idea will have more applications in the future. 

It would not be surprising to us if the coupling in Theorem~\ref{thm-main-intro} is unique in law. We will not address this question here, but we expect that the conformal removability of SLE$_4$-type curves~\cite{kms-removability} might be relevant (see~\cite[Theorem 1.4]{shef-zipper} for some discussion of why conformal removability is related to the uniqueness of conformal weldings of LQG surfaces). We have no reason to believe that Theorem~\ref{thm-cle4-coupling} can be extended to get similar couplings of supercritical LQG with CLE$_\kappa$ for $\kappa\not=4$; it would be surprising to us if such couplings exist.

\begin{remark} 
Prior to this work, it was not clear what was the correct way to generalize the relationships between SLE and LQG to the case when $\ccL\in (1,25)$, or indeed whether any such generalization should exist. 
A natural way to go about finding such a generalization is to first look for an appropriate generalization of SLE with central charge in $(1,25)$, then sample such a generalized SLE independently from the supercritical LQG. 
Some potential candidates for ``SLE with central charge in $(1,25)$'' are Loewner evolution driven by complex-valued Brownian motion~\cite{gp-complex-sle} or various objects associated with LQG with central charge $26-\ccL$ (motivated by~\eqref{eqn-matched}). 
It is still unclear whether any of these objects have a nice relationship to LQG with central charge $\ccL$. 
However, as we show in this paper, instead of generalizing SLE one can continue to look at an SLE$_4$-type object (as in the critical case $\ccL = 25$) and change the way that this SLE is coupled to LQG.
\end{remark}

\subsection{Applications and additional perspectives}

The relationships between SLE and LQG in the subcritical and critical phases have a huge number of applications, to such topics as SLE and of LQG individually, conformal field theory (see, e.g.,~\cite{as-cle-integrability,arsz-structure-constants}), random planar maps (see, e.g.,~\cite{ghs-map-dist,gm-spec-dim,gms-tutte}), random permutations (see, e.g.,~\cite{borga-skew-permuton,bgs-meander}), percolation on the triangular lattice (see, e.g.,~\cite{nqsz-backbone}), and the moduli of random surfaces~\cite{ars-annulus}. Many of these applications are surveyed in~\cite{ghs-mating-survey}.

Likewise, Theorem~\ref{thm-main-intro} opens the door to a much deeper understanding of supercritical LQG and its connections to other objects. 
For example, our result provides new insights regarding extensions to the supercritical phase of the other three items above:
\begin{enumerate}
\setcounter{enumi}{1}
\item Our supercritical LQG disk can (at least heuristically) be described in terms of a version of the Liouville action (Section~\ref{sec-action}). Additionally, correlation functions for the supercritical LQG disk are likely to be explicitly computable, see Problem~\ref{prob-critical-corr} and the surrounding discussion. 
\item We give the first scaling limit conjecture for a class of combinatorially natural random planar maps (i.e., one not defined in terms of continuum objects) toward supercritical LQG (Section~\ref{sec-rpm-supercritical}). The justification for the conjecture is that the random planar maps exhibit discrete analogs of various properties of our coupling of supercritical LQG with CLE$_4$. We expect that similar ideas can also lead to connections between other types of random planar maps and supercritical LQG, see Remark~\ref{remark-rpm}.
\item Our LQG disk with central charge $\ccL\in(1,25)$ admits a canonical finite length measure on $\bdy \BB D$. This length measure is also defined on the CLE$_4$ loops in the coupling of Theorem~\ref{thm-cle4-coupling}, although the measures on the inside and outside of the loops do not agree (see Section~\ref{sec-gap}). The joint law of the supercritical LQG lengths of the CLE$_4$ loops has an explicit description, see Section~\ref{sec-boundary-law}.
\end{enumerate}
The recent paper~\cite{bgs-supercritical-crt} builds on the present paper to prove several fundamental results about supercritical LQG. 
\begin{itemize}
\item \cite[Theorem 1.5]{bgs-supercritical-crt} shows that there is no reasonable notion of area measure associated with supercritical LQG (in contrast to the subcritical case). More precisely, there does not exist an area measure which is locally finite, locally determined by the field $\Phi$, and compatible with the supercritical LQG coordinate change rule (Definition~\ref{def-lqg-surface}). The proof is based on the coupling of supercritical LQG and CLE$_4$ introduced in the present paper.
\item Various works from physics, dating back to the 1980s, predict that supercritical LQG surfaces should behave like ``branched polymers'', i.e., they should look like the continuum random tree~\cite{aldous-crt1}. See~\cite[Section 1.1]{bgs-supercritical-crt} or~\cite[Section 2.2]{ghpr-central-charge} for discussion and references related to this prediction. \cite[Theorem 1.4]{bgs-supercritical-crt} is the first rigorous result which reconciles the ``branched polymer'' prediction with the description of supercritical LQG from recent mathematical literature (as discussed above). More precisely, this theorem shows that, for the random planar map models of supercritical LQG introduced in this paper, if one conditions on the rare event that the map is finite, then the scaling limit of the map is the continuum random tree.
\end{itemize}

We discuss some additional potential applications and extensions of the results of the present paper in Section \ref{sec-open-problems}.

\begin{remark}[Physics motivation] \label{remark-string} 
Polyakov's motivation for introducing LQG in~\cite{polyakov-qg1} comes from bosonic string theory. 
In this theory, a string is a loop (continuous image of the circle) in $\BB R^d$, $d \in \BB N_0$, which evolves in time. 
A string can therefore be viewed as a continuous function from a two-dimensional parameter space into $\BB R^d$, with the two parameters corresponding to the parametrization of the loop and time. 
Equivalently, we can think of a string as a surface together with a function from the surface into $\BB R^d$. This object is sometimes called a \textbf{worldsheet}. 

To analyze such objects, Polyakov sought to develop a theory of ``integrating over all possible surfaces in $\BB R^d$'', analogous to the Feynman path integral, which is an integral over all possible paths. 
To make sense of this, one needs a measure on the set of pairs consisting of a surface and a function from the surface into $\BB R^d$. 
Polyakov argued that the right measure to put on such surface/function pairs is the one where the surface is LQG with central charge $\ccL = 26-d$, and the ``function'' (actually a generalized function) is given by $d$ independent Gaussian free fields (GFFs) on the surface. The weighting $(\det \Delta_g)^{-(26-\ccL)/2}$ in the heuristic definition of LQG can be viewed as the partition function of $26-\ccL$ independent GFFs. 

If $d\in\BB N_0$, we have that $\ccL$ is in the subcritical / critical phase $\ccL \geq 25$ if and only if $d \in \{0,1\}$. For $d\in \{2,\dots,24\}$, instead $\ccL \in (1,25)$ is in the supercritical phase. 
 This suggests that the supercritical phase of LQG should be connected to $d$-dimensional string theory for $d \in \{2,\dots,24\}$. 
Moreover, one would like to understand supercritical LQG of central charge $\ccL \in \{1,\dots,24\}$ together with $d = 26-\ccL$ independent GFFs. 
We do not prove any results about this particular combination of objects in the present paper.
See, e.g.,~\cite{dgz-2d-gravity-matrix,gm-2d-gravity-string} for expository articles covering the relations between LQG and string theory from a physics perspective. 

It is also possible to express Wilson loop observables in quantum Yang-Mills theory in terms of ``sums over all possible surfaces in $\BB R^d$''. See, e.g.,~\cite{chatterjee-lattice-gauge,mp-matrix-group,cps-random-surface-ym} and the references therein for some recent rigorous results in this direction for lattice Yang-Mills theory. It is therefore conceivable, but still far from rigorously established, that there could be connections between Yang-Mills theory and supercritical LQG. See the open problems in~\cite{cps-random-surface-ym} for further discussion.
\end{remark}

\section{The supercritical LQG disk and its coupling to CLE$_4$}
\label{sec-main}

This section is the core of the paper: we will give our definition of the supercritical LQG disk and state and prove our main theorem (Theorem~\ref{thm-main-intro}) concerning its coupling with CLE$_4$.
In Section~\ref{sec-lqg-surface}, we review the definitions of LQG surfaces (Definition~\ref{def-lqg-surface}) and the critical LQG disk (Definition~\ref{def-critical-disk}).
In Section~\ref{sec-supercritical-disk}, we state our definition of the supercritical LQG disk (Definition~\ref{def-supercritical-disk}).
In Section~\ref{sec-cle}, we review some properties of conformal loop ensembles and state a precise version of Theorem~\ref{thm-main-intro} (Theorem~\ref{thm-cle4-coupling}). 
In Section~\ref{sec-cle4-coupling}, we define a coupling of the supercritical LQG disk with \emph{nested} CLE$_4$ (Definition~\ref{def-disk-cle4}) and prove that it satisfies a Markov property which is stronger than the property described in Theorem~\ref{thm-main-intro}. 
In Section~\ref{sec-comments}, we make some observations about our coupling and its relationship to other objects.
In Section~\ref{sec-non-trivial}, we prove that in our coupling the CLE$_4$ is neither independent from nor determined by the supercritical LQG disk.

\subsection{LQG surfaces}
\label{sec-lqg-surface}

In many results concerning LQG, including couplings to SLE, connections to random planar maps, and exact formulas in the setting of conformal field theory, one has to work with a certain canonical LQG surface with a given topology. 
In the subcritical and critical cases, such canonical LQG surfaces are defined, e.g., in~\cite{wedges,dkrv-lqg-sphere,hrv-disk,grv-higher-genus}. In this subsection we define a canonical LQG surface with the disk topology in the supercritical case $\ccL \in (1,25)$. The reader who just wants to see the definition, without the background and motivation, can skip to Definition~\ref{def-supercritical-disk}. 

We first recall the definition of LQG surfaces from~\cite{shef-kpz,shef-zipper,wedges}, which is usually stated only for $\ccL \geq 25$ but also makes perfect sense for all $\ccL  > 1$. In what follows, when we talk about a generalized function (distribution) on $U\subset\BB C$ we mean a continuous linear functional $\phi :  C_c^\infty(U) \to \BB R$. We usually write the action of $\phi$ on $f$ as $\phi(f) = \int_U \phi(z) f(z) \,d^2z$. For the purposes of defining $\sigma$-algebras, we equip the space of generalized functions with the topology of weak convergence: $\phi_n\to \phi$ if $\int_U \phi_n(z) f(z) \,dz \to \int_U \phi(z) f(z)\,dz$ for each $f\in C_c^\infty(U)$.

\begin{defn} \label{def-lqg-surface}
Let $\ccL  > 1$ and let $Q = \sqrt{\frac{\ccL-1}{6}}>0$ be the corresponding background charge. 
Let $k\geq 0$ and consider the set of $k+2$-tuples $(U,\Phi,x_1,\dots,x_k)$, where $U\subset \BB C$ is open, $\Phi$ is a generalized function on $U$ (which we will always take to be random, and in particular to be some variant of the GFF), and $x_1,\dots,x_k\in U\cup \bdy U$. We define an equivalence relation ${\sim_Q}$ on such $k+2$-tuples by
\eqb \label{eqn-lqg-equiv}
(U,\Phi,x_1,\dots,x_k) \sim_Q (\wt U,\wt\Phi, \wt x_1,\dots,\wt x_k)
\eqe 
if there is a conformal map $f : \wt U \to U$ such that 
\eqb \label{eqn-lqg-coord}
\wt\Phi = \Phi\circ f +Q\log|f'| \quad \text{and} \quad f(\wt x_j) = x_j,\quad\forall j=1,\dots,k .
\eqe  
A \textbf{Liouville quantum gravity (LQG) surface} of central charge $\ccL$, with $k$ marked points, is an equivalence class under ${\sim_Q}$.
If $(U,\Phi,x_1,\dots,x_k)$ is an equivalence class representative of an LQG surface $\mcl S$, we refer to $\Phi$ as an \textbf{embedding} of $\mcl S$ into $(U,x_1,\dots,x_k)$. 
\end{defn}

We think of two equivalent $k+2$-tuples as in~\eqref{eqn-lqg-equiv} as two different parametrizations of the same LQG surface.
Objects associated with LQG are required to be compatible with the LQG coordinate change formula. For example, if $\ccL \geq 25$ and if $\mu_\Phi = ``e^{\gamma \Phi(z)} \,d^2 z''$ is the LQG area measure (defined via Gaussian multiplicative chaos), then whenever $f , \Phi$, and $\wt\Phi$ are as in~\eqref{eqn-lqg-coord}, one has $\mu_\Phi = f_* \mu_{\wt\Phi}$~\cite{shef-kpz,shef-wang-lqg-coord} (see~\cite[Theorem 13]{shef-renormalization} for the critical case).
The analogous statement also holds for LQG boundary length measures.

We henceforth assume that the reader is familiar with the free boundary GFF (a.k.a.\ the GFF with Neumann boundary conditions). This is a Gaussian random generalized function on a domain in $\BB C$, viewed modulo additive constant. The covariance kernel depends on the choice of additive constant, but one possible choice of covariance kernel for the free boundary GFF on $\BB D$ is 
\eqb \label{eqn-gff-cov}
\log\frac{1}{|z-w| |1-z\ol w|} , \quad z,w \in \BB D ,\quad z\not= w. 
\eqe 
See, e.g.,~\cite[Chapter 6]{bp-lqg-notes} for background on the free boundary GFF.  

Our definition of the supercritical LQG disk is based on the definition of the critical LQG disk, which we now recall. 
It is more convenient to state the definition when the LQG surface is parametrized by the infinite strip, rather than the disk. 

\begin{defn} \label{def-strip-decomp}
Let 
\eqbn
\mcl S = \BB R \times (0,\pi)
\eqen
and let $\wt\Phi$ be a free-boundary GFF on $\mcl S$, with any choice of additive constant (recall that a priori the free boundary GFF is defined modulo additive constant). 
For $u \in \BB R$, we define $\wt\Phi^| : \mcl S \to \BB R$ to be the random function which, on each vertical segment $\{x\}\times (0,\pi)$, is identically equal to the average of $\wt\Phi$ on this segment.
We define the \textbf{lateral part} of $\wt\Phi$ by $\wt\Phi^\dagger = \wt\Phi - \wt\Phi^|$. 
\end{defn}

By~\cite[Lemma 4.2]{wedges}, the random function $\wt\Phi^|$ (viewed modulo additive constant) and the random generalized function $\wt\Phi^\dagger$ are independent. We also note that $\wt\Phi^\dagger$ does not depend on the choice of additive constant for $\wt\Phi$. Furthermore, as explained in~\cite[Section 4.1.6]{wedges}, it follows from the calculations in~\cite[Section 6]{shef-kpz} that the function which takes $x$ to the value of $\wt\Phi^|$ on $\{x\} \times (0,\pi)$ has the same law as $\{B_{2x}\}_{x\in\BB R}$, where $B$ is a standard two-sided Brownian motion (viewed modulo a global additive constant).
The following definition of the critical LQG disk is taken from~\cite[Definition 4.3]{ahps-critical-mating}. 

\begin{defn}[Critical LQG disk] \label{def-critical-disk}
Let $\mcl B :\BB R \to (-\infty,0]$ be a random function such that $\{\mcl B_s\}_{s\geq 0}$ and $\{\mcl B_{-s}\}_{s\geq 0}$ are independent and each has the law of $(-\sqrt 2)$ times a three-dimensional Bessel process started from 0. Equivalently, $\mcl B$ is $\sqrt 2$ times a two-sided Brownian motion with $\mcl B_0 = 0$, conditioned to be negative.
With $\mcl S$ as in Definition~\ref{def-strip-decomp}, let $\Phi_0^|  : \mcl S \to \BB R$ be the random generalized function which is identically equal to $\mcl B_x$ on each vertical segment $\{x\}\times (0,\pi)$. 
Also let $\wt\Phi^\dagger$ be the lateral part of a free-boundary GFF on $\mcl S$, as in Definition~\ref{def-strip-decomp}, sampled independently from $\Phi_0^|$ and let 
\eqbn
\Phi_0 = \Phi_0^| + \wt\Phi^\dagger
\eqen
\begin{itemize}
\item Let $\nu_{\Phi_0} = \text{``}e^{\Phi_0(z)} \, |dz|\text{''}$ be the $\gamma=2$-LQG boundary length measure associated with $\Phi_0$ and let 
\eqbn
\Phi := \Phi_0 - \log \nu_{\Phi_0}(\bdy \mcl S) . 
\eqen 
For $L > 0$, the \textbf{doubly marked critical ($\ccL =25$, $\gamma=2$) LQG disk with boundary length $L$} is the critical LQG surface $(\mcl S , \Phi + \log L , -\infty,+\infty) / {\sim_2}$.  
\item The \textbf{singly marked (resp.\ unmarked) critical LQG disk with boundary length $L$} is the critical LQG surface obtained by forgetting one (resp.\ both) marked points. 
\end{itemize}
\end{defn}
 
We emphasize that by definition, multiplying the boundary length of a critical LQG disk by $L$ corresponds to adding $\log L$ to the corresponding field.

\begin{defn}[Standard embedding] \label{def-circle-avg}
The random generalized function $\Phi$ of Definition~\ref{def-critical-disk} is called the \textbf{standard embedding} of the doubly marked critical LQG disk into $(\mcl S , -\infty,\infty)$. It is characterized by the condition that $\Phi^|(x)$ attains its maximum value at $x=0$. The \textbf{standard embedding} of the doubly marked LQG disk into $(\BB D , -1 , 1)$ is the random generalized function $\Phi\circ f + 2\log|f'|$, where $f : \BB D\to\mcl S$ is the conformal map taking $\pm 1$ to $\pm\infty$ and $-i$ to $0$. A \textbf{standard embedding} of a singly marked or unmarked critical LQG disk is the field obtained from the standard embedding of a doubly marked critical LQG disk by forgetting one or both marked points.
\end{defn}

\begin{remark} \label{remark-subcritical-disk}
In the subcritical case $\ccL > 25$, the LQG disk is defined similarly except that we take $\mcl B$ to be a Brownian motion with negative drift conditioned to stay negative, instead of a 3d Bessel process, and we weight the law of $\Phi_0$ by $\nu_{\Phi_0}(\bdy\mcl S)^{4/\gamma^2-1}$, see~\cite[Section 4.5]{wedges}. There is also an alternative definition of the LQG disk given in~\cite{hrv-disk}, which is proven to be equivalent to the one in~\cite{wedges} in~\cite{cercle-quantum-disk} (see also~\cite{ahs-integrability} for an alternative proof).  
\end{remark}

One can produce a doubly marked critical LQG disk from an unmarked critical LQG disk by sampling the marked boundary points from the LQG measure, as follows.

\begin{lem} \label{lem-marked-pts}
The marked points of the doubly marked critical LQG disk of boundary length $L$ are uniformly random in the following sense. If $(\mcl S , \Phi)$ is an embedding of an unmarked critical LQG disk and conditional on $\Phi$ we sample $x,y\in \bdy \mcl S$ independently from the critical LQG length measure $\nu_\Phi$, then  $(\mcl S , \Phi,x,y) / {\sim_2}$ is a doubly marked critical LQG disk of boundary length $L$. 
\end{lem}
\begin{proof}
The analogous statement for a subcritical LQG disk is~\cite[Proposition A.8]{wedges}. The statement for the critical LQG disk follows from this and the convergence of subcritical LQG disks to critical LQG disks as $\gamma\to 2^-$~\cite[Lemma 4.5]{ahps-critical-mating}.  
\end{proof}

\subsection{The supercritical LQG disk}
\label{sec-supercritical-disk}

We now consider how to extend the definition of the critical unit boundary length LQG disk to the supercritical case.
It follows from Definition~\ref{def-critical-disk} that there is an embedding $\Phi$ of the critical LQG disk into $\BB D$ which can be written as $\Phi = \wt\Phi + f$, where $\wt\Phi$ is a free-boundary GFF on $\BB D$ and $f$ is a random function on $\ol{\BB D}$ (not necessarily independent from $\wt\Phi$) which is continuous except possibly at finitely many points of $\bdy\BB D$. Indeed, the field $\Phi$ on $\mcl S$ from Definition~\ref{def-critical-disk} has this property with $\mcl S$ in place of $\BB D$ since the function $\wt\Phi^|$ is continuous. The desired embedding into $\BB D$ is obtained by conformally mapping $\mcl S$ to $\BB D$. 

Recall that a point $z\in\bdy\BB D$ is an $\alpha$-thick point of $\wt\Phi$ if $\limsup_{\ep\to 0} \wt\Phi_\ep(z) / \log\ep^{-1} = \alpha$, where $\wt\Phi_\ep(z)$ is the average of $\wt\Phi$ on the almost-semicircle $\bdy B_\ep(z) \cap \BB D$. 
Almost surely, the free-boundary GFF $\wt\Phi$ has $\alpha$-thick points on $\bdy\BB D$ for any $\alpha \in [-2,2]$, but not for $|\alpha| > 2$ (this can be proven, e.g., via the same arguments as in~\cite{hmp-thick-pts}, which proves an analogous statement for thick points in the bulk). 
For $\ccL \in (1,25)$, we have $Q\in (0,2)$, and $\alpha$-thick points of the field $\wt\Phi$ for $\alpha>Q$ give rise to ``singular points'' for the supercritical LQG metric, i.e., points which lie at infinite distance from every other point (see~\cite{pfeffer-supercritical-lqg} for the case of interior thick points, we expect that similar statements hold for boundary thick points).  
This suggests that it is not possible to define a finite supercritical LQG boundary length measure associated with the free-boundary GFF. 
Hence, it is not obvious how to generalize the definitions of the unit boundary length LQG disk from~\cite{wedges,hrv-disk} to the case when $\ccL\in (1,25)$. 

One possible strategy is to start by modifying the free-boundary GFF $\wt\Phi$ so that it no longer has $\alpha$-thick points on the boundary for $\alpha>Q$. A very naive way to do this is as follows. We start by writing $\wt\Phi = \wt\Phi^0 + \phi$, where $\wt\Phi^0$ is a zero-boundary GFF on $\BB D$ and $\phi$ is a random harmonic function on $\BB D$ independent from $\wt\Phi^0$. 
The set of boundary thick points depends only on the function $\phi$, so $\wt\Phi^0 + (Q/2) \phi$ has boundary thick points for $\alpha \in [-Q,Q]$, but not for $|\alpha| > Q$, and it still has similar behavior to the free-boundary (or zero-boundary) GFF away from the boundary. Hence, it is reasonable to search for a supercritical analog of the LQG disk which locally behaves like $\wt\Phi^0  + (Q/2) \phi$ near the boundary. 

We note that if $\Psi$ is a zero-boundary GFF on $\BB D$ independent from $\wt\Phi$ and $a,b\in [-1,1]$ with $a^2+b^2=1$, then $a\wt\Phi^0 + b\Psi \eqD \wt\Phi^0$. Hence 
\eqbn
\wt\Phi^0 + \frac{Q}{2} \phi \eqD \frac{Q}{2} \wt\Phi  +  \frac{\sqrt{4-Q^2}}{2}  \Psi .
\eqen
This at least partially motivates the following definition. 

\begin{defn} \label{def-supercritical-disk}
Fix $Q\in (0,2)$ and let $\ccL = 1+6Q^2 \in (1,25)$ be the corresponding central charge. Let $L > 0$ and let $  \Phi_2 $ be the standard embedding into $\BB D$ of a doubly marked critical LQG disk with boundary length $L$ (Definitions~\ref{def-critical-disk} and~\ref{def-circle-avg}).
Let $\Psi$ be a zero-boundary GFF on $\BB D$, sampled independently from $ \Phi_2$, and let
\eqb \label{eqn-supercritical-disk}
\Phi := \frac{Q}{2} \Phi_2 + \frac{\sqrt{4-Q^2}}{2} \Psi .
\eqe 
The \textbf{(supercritical) LQG disk with central charge $\ccL$ and boundary length $L$} is the LQG surface $(\BB D,\Phi , -1,1) / {\sim_Q}$ (this is an LQG surface of central charge $\ccL$ since we are viewing it modulo LQG coordinate change with parameter $Q$, recall Definition~\ref{def-lqg-surface}). We call the field $\Phi$ its \textbf{standard embedding}. The \textbf{singly marked} or \textbf{unmarked} LQG disk with central charge $\ccL$ and boundary length $L$ is obtained by forgetting one or both marked points.
\end{defn}

We use the standard embedding of $\Phi_2$ in Definition~\ref{def-supercritical-disk} only for convenience. 
Using the conformal invariance of the law of the GFF $\Psi$, one can check that the law of the LQG surface $(\BB D ,\Phi , -1,1) / {\sim_Q}$ does not depend on the particular choice of embedding $\Phi_2$. 

At first glance, our definition of the supercritical LQG disk may not look particularly natural. The same is true of the definitions of the subcritical or critical LQG disks. However, there are various ways to justify why the definitions are the correct ones, which apply in both the subcritical and supercritical cases. 
\begin{itemize}
\item The subcritical and critical LQG disks arises naturally in couplings of SLE with LQG~\cite{wedges,ahps-critical-mating}. Our main theorem (Theorem~\ref{thm-cle4-coupling}) says that the supercritical LQG disk also arises naturally in couplings with SLE.  
\item The subcritical LQG disk can be viewed as a sample from a probability measure on fields defined in terms of the Liouville action. This is rigorously justified in~\cite{hrv-disk}. We give a heuristic interpretation of the supercritical LQG disk in terms of a variant of the Liouville action in Section~\ref{sec-action}.  
\item The subcritical and critical LQG disks arise, at least conjecturally, as the scaling limit of various types of random planar maps. See, e.g.,~\cite{hrv-disk,ghs-mating-survey} for conjectures and~\cite{bet-mier-disk,lqg-tbm2,gms-tutte,hs-cardy-embedding} for rigorous results. In Section~\ref{sec-rpm-supercritical}, we give a precise scaling limit conjecture for random planar maps toward the supercritical LQG disk.
\item As explained in Appendix~\ref{sec-disk-rotate}, one can ``rotate'' a $k$-tuple of independent fields sampled from the infinite measures on unmarked subcritical LQG disks (or spheres) by a $k\times k$ orthogonal matrix to obtain a $k$-tuple of independent fields corresponding to subcritical LQG disks (or spheres) with the same total central charge.
The field $\Phi$ in Definition~\ref{def-supercritical-disk} is obtained as a linear combination of the critical LQG disk field $\Phi_2$ and the zero-boundary GFF $\Psi$, i.e., $\Phi$ is obtained by ``rotating'' two fields.
Hence, Definition~\ref{def-supercritical-disk} is natural if we interpret the zero-boundary Gaussian free field as the ``LQG disk with central charge $\ccL = 1$''. This is reasonable, e.g., since the Gaussian free field defines a conformal field theory of central charge 1 (see Footnote~\ref{footnote-free-boundary} for some additional discussion of why we want a zero-boundary GFF instead of a free-boundary GFF).
\end{itemize}

Let $\Phi$ and $\Phi_2$ be as in Definition~\ref{def-supercritical-disk}.
To define a boundary length measure associated with $\Phi$, let $f : \BB R \to [0,\infty)$ be a smooth, compactly supported bump function with total integral one, let $f^\ep(\cdot) := \ep^{-1} f(\ep^{-1}\cdot)$, and define the convolution $\Phi^\ep(\cdot) := \frac{1}{2\pi} \int_{-\pi}^\pi f^\ep(\cdot - t) \Phi(e^{i t}) \,dt$, with the integral interpreted as a distributional pairing. 
We similarly define $ \Phi_2^\ep(z)$. By Definition~\ref{def-supercritical-disk}, the restrictions of $\Phi$ and $\frac{Q}{2} \Phi_2$ to $\bdy\BB D$ coincide, so also $\Phi^\ep|_{\bdy\BB D} = \frac{Q}{2} \Phi_2^\ep|_{\bdy\BB D}$.
We associate with $\Phi$ the \textbf{supercritical LQG boundary length measure} which is the critical Gaussian multiplicative chaos measure defined by the following limit in probability with respect to the weak topology for measures on $\bdy\BB D$: 
\eqb \label{eqn-supercritical-length} 
\nu_\Phi := \lim_{\ep \to 0} \sqrt{\log(1/\ep)} \ep e^{\frac{2}{Q} \Phi^\ep(z)} \,|dz|  = \lim_{\ep \to 0} \sqrt{\log(1/\ep)} \ep e^{ \Phi_2^\ep(z)} \,|dz|
\eqe 
where $|dz|$ denotes Lebesgue measure on $\bdy\BB D$. See~\cite[Theorem 2.5]{powell-gmc-survey} and the references thereafter for the existence of the limit in a much more general setting. In other words, the supercritical LQG boundary length measure associated with $\Phi$ is the same as the critical LQG length measure associated with $\Phi_2$. 

\begin{lem} \label{lem-bdy-coord}
The measure $\nu_\Phi$ is compatible with the $Q$-LQG coordinate change formula, in the sense that if $f : \BB D \to \BB D$ is a conformal automorphism, then a.s.\ 
\eqb \label{eqn-supercritical-coord} 
\nu_\Phi(A) = \nu_{\Phi\circ f + Q\log|f'|}(f^{-1}(A) ) ,\quad \text{$\forall A\subset \bdy \BB D$ Borel} . 
\eqe 
\end{lem}
\begin{proof}
We have 
\eqb \label{eqn-coord-compatible}
\Phi \circ f + Q \log|f'|  = \frac{Q}{2} \left( \Phi_2 \circ f + 2\log |f'| \right) + \frac{\sqrt{4-Q^2}}{2}\Psi \circ f   .
\eqe 
Since $\Psi \circ f$ is zero on $\bdy\BB D$, the coordinate change relation~\eqref{eqn-supercritical-coord} follows from this and the LQG coordinate change formula for the critical LQG boundary length measure.
The LQG coordinate change formula for the critical LQG boundary length measure can be obtained in various ways, e.g., by adapting the argument of~\cite[Theorem 13]{shef-renormalization} (which gives the analogous statement for the critical LQG area measure) or by taking a limit of the $\gamma$-LQG boundary length measures as $\gamma \to 2^-$~\cite[Section 4.1]{aps-critical-lqg-lim}.
\end{proof}

\begin{remark} \label{remark-bdy-coord}
In light of Lemma~\ref{lem-bdy-coord}, if $\Phi$ is an embedding of a supercritical LQG disk into a domain $U\subset\BB C$ whose boundary is the image of a curve, then we can define the LQG length measure $\nu_\Phi$ on $\bdy U$ to be the pushforward under $f$ of $\nu_{\Phi\circ f + Q \log|f'|}$, where $f : \BB D \to U$ is a conformal map. As mentioned earlier, it is shown in~\cite[Theorem 1.5]{bgs-supercritical-crt} that there is no natural area measure associated with a supercritical LQG surface.
\end{remark} 

\subsection{Conformal loop ensembles and their couplings with the GFF} 
\label{sec-cle}

\subsubsection{Non-nested CLE and LQG} 
\label{sec-non-nested}

For $\kappa \in (8/3,4]$, the \textbf{(non-nested) conformal loop ensemble (CLE$_\kappa$)} on a proper simply connected open set $U\subset\BB C$ is a random countable collection of disjoint non-nested loops in $U$ which each locally look like SLE$_\kappa$ curves. Almost surely, Lebesgue-a.e.\ point of $U$ is surrounded by exactly one loop. The law of CLE$_\kappa$ is conformally invariant in the sense that if $\Gamma$ is a CLE$_\kappa$ in $U$ and $f : U\to V$ is a conformal map, then $f(\Gamma)$ is a CLE$_\kappa$ in $V$. CLE$_\kappa$ can be defined via branching SLE$_\kappa(\kappa-6)$ curves~\cite{shef-cle} or equivalently as the outer boundaries of the outermost clusters of a Brownian loop soup on $U$ with intensity $\cc(\kappa)/2$~\cite{shef-werner-cle}. 

The following is a precise version of Theorem~\ref{thm-main-intro}.

\begin{thm}[LQG/CLE$_4$ coupling, supercritical case] \label{thm-cle4-coupling} 
Let $\ccL \in (1,25)$, let $Q = \sqrt{(\ccL-1)/6}$, and let $\Phi$ be the standard embedding into $\BB D$ of a doubly marked unit boundary length supercritical LQG disk with central charge $\ccL$ (Definition~\ref{def-supercritical-disk}). There exists a coupling $(\Phi ,\Gamma)$ of $\Phi$ with a (non-nested) CLE$_4$ $\Gamma$ such that the following is true. 
For each loop $\ell \in \Gamma$, let $U_\ell \subset \BB D$ be the open region enclosed by $\ell$. 
Then the central charge $\ccL$ LQG surfaces $\{ (U_\ell , \Phi|_{U_\ell}) / {\sim_Q} \}_{ \ell\in \Gamma }$ are conditionally independent LQG disks with central charge $\ccL$ given their supercritical LQG boundary lengths (recall from Definition~\ref{def-supercritical-disk} that a supercritical LQG disk comes with an LQG boundary length). 
Furthermore, in this coupling, $\Gamma$ is neither independent from $\Phi$ nor determined by $\Phi$.
\end{thm}

In the setting of Theorem~\ref{thm-cle4-coupling}, the boundary lengths of the LQG surfaces $(U_\ell , \Phi|_{U_\ell}) / {\sim_Q} $ are well-defined since these LQG surfaces are supercritical LQG disks, see~\eqref{eqn-supercritical-length} and the discussion just after. 

We will actually prove a stronger version of Theorem~\ref{thm-cle4-coupling}, namely Theorem~\ref{thm-cle4-coupling-nested}, where we couple $\Phi$ with a whole nested CLE$_4$. Theorem~\ref{thm-cle4-coupling} is obtained from this result by forgetting all but the outermost loops. 

Theorem~\ref{thm-cle4-coupling} is the supercritical analog of the following theorem for subcritical and critical LQG. 

\begin{thm}[LQG/CLE coupling, subcritical/critical case] \label{thm-subcritical-coupling}
Let $\gamma \in (\sqrt{8/3},2]$, equivalently $Q = 2/\gamma + \gamma/2 \in [2,5/\sqrt 6)$ or $\ccL = 1+6Q^2 \in [25,26)$. 
Let $\Phi$ be the standard embedding into $\BB D$ of a doubly marked unit boundary length $\gamma$-LQG disk (defined in the same manner as in Definitions~\ref{def-critical-disk} and~\ref{def-circle-avg}, see Remark~\ref{remark-subcritical-disk}).
Let $\Gamma$ be a (non-nested) CLE$_{\kappa=\gamma^2}$ in $\BB D$ sampled independently from $\Phi$. 
For each loop $\ell \in \Gamma$, let $U_\ell \subset \BB D$ be the open region enclosed by $\ell$. 
Then the central charge $\ccL$ LQG surfaces $\{ (U_\ell , \Phi|_{U_\ell}) / {\sim_Q} \}_{ \ell\in \Gamma }$ are conditionally independent LQG disks of central charge $\ccL$ given their boundary lengths.  
\end{thm}

In the subcritical case $\gamma \in (\sqrt{8/3},2)$, Theorem~\ref{thm-subcritical-coupling} follows from~\cite[Theorem 1.1]{msw-simple-cle-lqg}. 
The critical case $\gamma=2$ can be deduced from the subcritical case by taking a limit as $\gamma\to 2^-$, see~\cite[Theorem 1.1]{ag-critical-cle4}.

\subsubsection{Nested CLE and the GFF} 
\label{sec-nested}

Let $U\subset\BB C$ be open and simply connected, $U\not=\BB C$. 
The \textbf{nested CLE$_\kappa$} $\ol\Gamma$ on $U$ is obtained from the non-nested CLE$_\kappa$ $\Gamma$ by the following inductive procedure.
Let $\Gamma^1 := \Gamma$. 
Inductively, assume that $n\in\BB N$ and a collection of loops $\Gamma^n$ in $U$ (the \textbf{$n$th generation loops}) has been defined. 
Conditional on $\{\Gamma^k\}_{k\leq n}$, for each loop $\ell\in \Gamma^n$, let $\Gamma_\ell$ be a CLE$_\kappa$ in the open region $U_\ell$ enclosed by $\ell$. We take these CLE$_\kappa$s to be conditionally independent given $\{\Gamma^k\}_{k\leq n}$. We then let $\Gamma^{n+1} = \bigcup_{\ell \in\Gamma^n} \Gamma_\ell$. Finally, we let 
\eqb \label{eqn-nested}
\ol\Gamma := \bigcup_{n=1}^\infty \Gamma^n .
\eqe 

A crucial property of CLE$_4$ for the purposes of this paper is its representation as the \textbf{level loops} of the zero-boundary GFF~\cite{shef-miller-cle4} (see \cite[Section 4]{asw-local-sets} for a published construction). 
The CLE$_4$ loops are not literal level loops of $\Psi$; rather, the field takes constant values on the inner and outer sides of each loop, and these values differ by a constant. Consequently, the CLE$_4$ loops are also called height-gap lines \cite{qw-common-level-lines} or two-valued sets \cite{aru-sepulveda-2valued} of $\Psi$.
To explain the construction, let $\ol\Gamma$ be a nested CLE$_4$ on $U\subset \BB C$, as in~\eqref{eqn-nested}. 
Let $\{X_\ell\}_{\ell\in \ol\Gamma}$ be random variables indexed by the CLE$_4$ loops which are conditionally independent given $\ol\Gamma$, with 
\eqb \label{eqn-cle-signs}
\BB P[X_\ell = 1\,|\,\ol\Gamma]  =\BB P[X_\ell = -1\,|\, \ol\Gamma] = \frac12 .
\eqe  
For $n\in\BB N$, let $\Psi_n$ be the piecewise constant, a.e.\ defined function on $\BB D$ which, for each $n$th generation loop $\ell\in\Gamma^n$, satisfies\footnote{ 
The constant $\pi$ in \eqref{eqn-restricted-gff} depends on the choice of normalization for the GFF $\Psi$. In this paper, as is common in papers working with Liouville quantum gravity, our normalization is chosen so that $\op{Cov}(\Psi(z) , \Psi(w)) \sim \log |z-w|^{-1}$ as $w\to z$ (c.f.~\eqref{eqn-gff-cov}). If instead $\op{Cov}(\Psi(z) , \Psi(w)) \sim c\log |z-w|^{-1}$, for some $c>0$, then $\pi$ would be replaced by $\sqrt c \pi$. 
}
\eqb \label{eqn-restricted-gff}
\Psi_n|_{U_\ell} = \pi \sum_{k=1}^n X_{\ell^{(k)}} 
\eqe
where $\ell^{(n)} = \ell$ and for $k=1,\dots,n-1$, $\ell^{(k)}$ is the unique loop in $\Gamma^k$ which disconnects $\ell$ from $\bdy\BB D$. Then $\Psi_n$ converges a.s.\ as $n\to\infty$ to a zero-boundary GFF $\Psi$ with respect to the topology of weak convergence for generalized functions.

In the above coupling, $\Psi$ and $\left(\ol\Gamma , \{X_\ell\}_{\ell\in\ol\Gamma} \right)$ are a.s.\ given by measurable functions of each other. 
Furthermore, by the recursive nature of the construction, we have the following Markov property. For each $n\in\BB N$, the conditional law of $(\Psi,\Gamma)$ given
\eqb \label{eqn-cond-on-outer}
\bigcup_{k=1}^n \Gamma^k  \quad \text{and} \quad \left\{ X_\ell : \ell\in \bigcup_{k=1}^n \Gamma^k\right\} 
\eqe 
is described as follows. 
Conditional on~\eqref{eqn-cond-on-outer}, let $\{\Psi_\ell\}_{\ell\in \Gamma^n}$ be conditionally independent zero-boundary GFFs on the domains $U_\ell$ enclosed by the loops in $\Gamma^n$.
Then 
\eqb \label{eqn-level-line-full}
\Psi|_{U_\ell} = \Psi_\ell + \pi \sum_{k=1}^n X_{\ell^{(k)}}   \quad \forall \ell \in \Gamma^n , 
\eqe 
where $\ell^{(k)}$ for $k=1,\dots,n$ are the loops surrounding $\ell$, as in~\eqref{eqn-restricted-gff}.
Furthermore, the set $\ol\Gamma|_{U_\ell}$ of loops of $\ol\Gamma$ which are contained in $U_\ell$ is the same as the set of level loops of $\Psi_\ell$.

\subsection{Coupling of supercritical LQG with nested CLE$_4$} 
\label{sec-cle4-coupling}

In this subsection we will define our coupling of the supercritical LQG disk with a nested CLE$_4$. 
In light of Definition~\ref{def-supercritical-disk} and the coupling of the zero-boundary GFF with CLE$_4$ in Section~\ref{sec-nested}, one might already have a guess as to what our coupling should look like. 
To give a precise definition, it is convenient to have the following terminology. 

\begin{defn}[Loop-decorated LQG surface] \label{def-lqg-surface-loop}
Let $\ccL  > 1$ and consider the set of triples $(U,\Phi,\Gamma)$, where $U\subset \BB C$ is open, $\Phi$ is a generalized function on $U$, and $\Gamma$ is a countable collection of loops in $U$. We define an equivalence relation on such triples by $(U,\Phi,\Gamma) \sim_Q (\wt U,\wt\Phi, \wt \Gamma)$ if there is a conformal map $f : \wt U \to U$ such that 
\eqb \label{eqn-lqg-coord-loop}
\wt\Phi = \Phi\circ f +Q\log|f'| \quad \text{and} \quad f(\wt \Gamma) = \Gamma
\eqe 
A \textbf{loop-decorated LQG surface} of central charge $\ccL$ is an equivalence class of such triples under ${\sim_Q}$.
If $(U,\Phi,\Gamma)$ is an equivalence class representative, we refer to $(\Phi,\Gamma)$ as an \textbf{embedding} of the corresponding loop-decorated LQG surface into $U$. 
\end{defn}

The following definition gives the coupling of supercritical LQG and CLE$_4$ considered in this paper. 

\begin{defn}[CLE$_4$-decorated LQG disk] \label{def-disk-cle4} 
Let $\ccL \in (1,25]$ and $L > 0$. 
Let $\Phi_2$ be the standard embedding into $\BB D$ of a critical LQG disk of boundary length $L$ (Definition~\ref{def-circle-avg}) and let $\Psi$ be a zero-boundary GFF on $\BB D$, sampled independently from $\Phi_2$. 
Let
\eqb \label{eqn-coupling-def}
\Phi = \frac{Q}{2} \Phi_2 + \frac{\sqrt{4-Q^2}}{2} \Psi  .
\eqe 
be a supercritical LQG disk of central charge $\ccL$ and boundary length $L$ as in Definition~\ref{def-supercritical-disk}. 

Let $\ol\Gamma$ be the nested CLE$_4$ which is the full set of level loops of $\Psi$, as in Section~\ref{sec-non-nested}. 
The \textbf{CLE$_4$-decorated LQG disk of central charge $\ccL$ and boundary length $L$} is the loop-decorated LQG surface $(\BB  D,\Phi,\ol\Gamma)/{{\sim_Q}}$. The pair $(\Phi,\ol\Gamma)$ is called a \textbf{standard embedding} of this loop-decorated LQG surface. Note that this embedding is not determined by $(\BB  D,\Phi,\ol\Gamma)/{{\sim_Q}}$.  
\end{defn}

We now verify that the coupling of Definition~\ref{def-disk-cle4} satisfies the desired Markov property. 
The Markov property is more complicated to state than in Theorem~\ref{thm-cle4-coupling} since we are working with a nested CLE$_4$. 

\begin{thm} \label{thm-cle4-coupling-nested}
Let $\ccL \in (1,25]$ and let $(\BB D,\Phi,\ol\Gamma)/{\sim_Q}$ be a CLE$_4$-decorated LQG disk of central charge $\ccL$ and unit boundary length, as in Definition~\ref{def-disk-cle4}.  
For $n\in\BB N$ let $\Gamma^n \subset\ol\Gamma$ be the set of $n$th generation loops as in~\eqref{eqn-nested}. 
For each loop $\ell \in \ol\Gamma$, let $U_\ell \subset \BB D$ be the open region enclosed by $\ell$ and let $\ol\Gamma|_{U_\ell}$ be the set of loops in $\ol\Gamma$ which are contained in $U_\ell$. 

For each $n\in\BB N$, if we condition on the supercritical LQG boundary lengths of the LQG surfaces
\eqbn
\left\{    (U_\ell , \Phi|_{U_\ell}) / {\sim_Q}   \,:\, \ell \in \bigcup_{k=1}^n \Gamma^k \right\}
\eqen 
then the loop-decorated central charge-$\ccL$ LQG surfaces 
\eqb \label{eqn-cle4-surfaces}
\{ (U_\ell , \Phi|_{U_\ell} , \ol\Gamma|_{U_\ell} ) / {\sim_Q} \}_{ \ell\in \Gamma^n }
\eqe 
are conditionally independent CLE$_4$-decorated LQG disks of central charge $\ccL$ with the given boundary lengths. 
\end{thm}

The conditioning in Theorem~\ref{thm-cle4-coupling-nested} is equivalent to conditioning on the inner supercritical LQG lengths of the CLE$_4$ loops in the first $n$ generations, c.f.\ Section~\ref{sec-gap}. Theorem~\ref{thm-cle4-coupling-nested} remains true if we condition only on the supercritical LQG boundary lengths of the surfaces $(U_\ell ,\Phi|_{U_\ell})/{\sim_Q}$ for $\ell \in \Gamma^n$ (instead of for all $\ell \in \bigcup_{k=1}^n \Gamma^k$): this is immediate from Theorem~\ref{thm-cle4-coupling-nested} since the conditional law in the theorem statement depends only on these boundary lengths. The statement we give (where we condition on the supercritical LQG lengths for all $\ell \in \bigcup_{k=1}^n \Gamma^k$) is stronger than the statement where we just condition on the lengths for $\ell\in\Gamma^n$.

We will prove in Section~\ref{sec-non-trivial} that in the setting of Theorem~\ref{thm-cle4-coupling-nested} for $\ccL\in (1,25)$, the CLE$_4$ $\ol\Gamma$ is neither independent from nor determined by $\Phi$. 
We will deduce Theorem~\ref{thm-cle4-coupling-nested} from the analogous statement in the critical case.

\begin{lem} \label{lem-cle4-coupling-critical}
The statement of Theorem~\ref{thm-cle4-coupling-nested} is true in the critical case $\ccL = 25$.
\end{lem}
\begin{proof}
In this case, we get from Definition~\ref{def-disk-cle4} that $\Phi = \Phi_2$ and $\ol\Gamma$ are independent. 
By Theorem~\ref{thm-subcritical-coupling} in the case $\gamma=2$ and the recursive definition of nested CLE$_4$ (see~\eqref{eqn-nested}), we get the Markov property of Theorem~\ref{thm-cle4-coupling-nested} in the case $n=1$. 
In general, the statement follows by induction on $n$. 
\end{proof}

\begin{proof}[Proof of Theorem~\ref{thm-cle4-coupling-nested}]  
Let $\Phi_2$, $\Psi$, and $\ol\Gamma$ be as in Definition~\ref{def-disk-cle4}. 
Basically, the theorem follows from the analogous coupling statement in the critical case (Lemma~\ref{lem-cle4-coupling-critical}), the Markovian property~\eqref{eqn-level-line-full} of the coupling of $\Psi$ with $\ol\Gamma$, and the fact that the linear combination~\eqref{eqn-coupling-def} behaves nicely under LQG coordinate change. Let us now give the details. 
\medskip
  
\noindent\textit{Step 1: constructing conditionally independent critical LQG disks and zero-boundary GFFs.} 
Let $\{X_\ell\}_{\ell\in\ol\Gamma}$ be the i.i.d.\ Bernoulli$(1/2)$ signs indexed by the loops of $\ol\Gamma$, as in~\eqref{eqn-cle-signs} and~\eqref{eqn-level-line-full}. 
For $n\in\BB N$, we define the $\sigma$-algebra 
\eqb \label{eqn-length-info}
\mcl F_n := \sigma \left\{ (  \nu_{\Phi_2}(\ell)  , X_\ell ) \,:\, \ell \in \bigcup_{k=1}^n \Gamma^k \right\} .
\eqe 
We first describe the conditional law of the loop-decorated LQG surfaces~\eqref{eqn-cle4-surfaces} given $\mcl F_n$. 

The signs $\{X_\ell\}_{\ell\in \ol\Gamma}$ are independent from $\Phi_2$, so by Lemma~\ref{lem-cle4-coupling-critical} we get that under the conditional law given $\mcl F_n$, the loop-decorated critical LQG surfaces 
\eqb \label{eqn-use-critical}
\{(U_\ell , \Phi_2|_{U_\ell} , \ol\Gamma|_{U_\ell} ) /{\sim_2} \}_{\ell\in\Gamma^n}
\eqe 
are conditionally independent critical CLE$_4$-decorated LQG disks with the given boundary lengths. 

To make use of this, conditional on $(\Phi_2,\Psi)$, for $\ell \in \Gamma^n$ let $x_\ell , y_\ell \in \ell$ be sampled from the critical LQG length measure $\nu_{\Phi_2}|_\ell$, normalized to be a probability measure. We take these points to be conditionally independent given $(\Phi_2,\Psi)$. By Lemma~\ref{lem-marked-pts}, $(U_\ell , \Phi_2|_{U_\ell}  ,x_\ell , y_\ell) / {\sim_2}$ is a doubly marked critical LQG disk conditional on its boundary length. 
For $\ell\in\Gamma^n$, let $f_\ell : \BB D \to U_\ell$ be the conformal map such that $f_\ell(-1) = x_\ell$, $f_\ell(1) = y_\ell$, and the field 
\eqb \label{eqn-map-to-disk2}
\Phi_2 \circ f_\ell +  2\log|f_\ell'| 
\eqe 
is the standard embedding into $\BB D$ of the critical LQG disk $(U_\ell , \Phi_2|_{U_\ell}  ,x_\ell , y_\ell) / {\sim_2}$ (Definition~\ref{def-circle-avg}).  

By~\eqref{eqn-use-critical}, the pairs 
\eqb \label{eqn-coupling-critical-cle4}
(\Phi_2 \circ f_\ell +  2\log|f_\ell'| , f_\ell^{-1}(\ol\Gamma|_{U_\ell}))
\eqe 
for different loops $\ell \in \Gamma^n$ are conditionally independent given $\mcl F_n$, and the conditional law of each is that of the standard embedding of a critical LQG disk with the given boundary length together with an independent CLE$_4$. 
 
The signs $\{X_{\ell'} : \ell'\in \ol\Gamma \}$ are conditionally independent given $(\Phi_2 ,\ol\Gamma)$.
From this and~\eqref{eqn-coupling-critical-cle4}, we infer that under the conditional law given $\mcl F_n$, the triples
\eqbn
\left( \Phi_2 \circ f_\ell +  2\log|f_\ell'| , f_\ell^{-1}(\ol\Gamma|_{U_\ell}) , \{X_{\ell'} : \ell'\in \ol\Gamma|_{U_\ell}  \} \right)
\eqen
are conditionally independent and the conditional law of each is that of the standard embedding into $\BB D$ of a critical LQG disk with the given boundary length together with an independent CLE$_4$ and a collection of independent Bernoulli$(1/2)$ signs associated with the CLE$_4$ loops. 
 
For $\ell\in\Gamma^n$, let $\Psi_\ell$ be the zero-boundary GFF on $U_\ell$ as in~\eqref{eqn-level-line-full}.
By the definition of the coupling of $\Psi$ with $\ol\Gamma$, we have that $\Psi_\ell\circ f_\ell$ is determined by $f_\ell^{-1}(\ol\Gamma|_{U_\ell}) $ and $\{X_{\ell'} : \ell'\in \ol\Gamma|_{U_\ell}  \}$ in the same manner that $\Psi$ is determined by $\ol\Gamma$ and $\{X_{\ell'} \}_{\ell' \in\ol\Gamma}$. 
By combining this with the previous paragraph, we get that under the conditional law given $\mcl F_n$, the triples
\eqb \label{eqn-coupling-triples}
\left( \Phi_2 \circ f_\ell +  2\log|f_\ell'| , \Psi_\ell \circ f_\ell , f_\ell^{-1}(\ol\Gamma|_{U_\ell})   \right)
\eqe 
are conditionally independent, and the conditional law of each is that of the standard embedding into $\BB D$ of a critical LQG disk with the given boundary length together with an independent zero-boundary GFF and its associated CLE$_4$ level lines.
\medskip

\noindent\textit{Step 2: conclusion.}  
Let us now consider the fields that we get from the central charge-$\ccL$ LQG coordinate change rule. Recalling~\eqref{eqn-coupling-def}, then using~\eqref{eqn-level-line-full}, we obtain that for each $\ell\in \Gamma^n$,
\allb \label{eqn-coupling-coord}
\Phi \circ f_\ell + Q \log |f_\ell'| 
&= \frac{Q}{2} (\Phi_2\circ f_\ell + 2\log|f_\ell'|)  + \frac{\sqrt{4-Q^2}}{2} \Psi \circ f_\ell \notag\\
&= \frac{Q}{2} \left( \Phi_2\circ f_\ell + 2\log|f_\ell'| + L_\ell \right)  + \frac{\sqrt{4-Q^2}}{2} \Psi_\ell \circ f_\ell  ,\notag\\
&\qquad\text{where} \quad L_\ell := \frac{\sqrt{4-Q^2}}{Q} \pi  \sum_{k=1}^n X_{\ell^{(k)}}  . 
\alle
Here, as in~\eqref{eqn-level-line-full}, $\ell^{(n)} =\ell$ and for $k=1,\dots,n-1$, $\ell^{(k)} \in \Gamma^k$ is the unique $k$th generation loop surrounding $\ell$. 

By the discussion surrounding~\eqref{eqn-coupling-triples} and the Definition~\ref{def-critical-disk} of the critical LQG disk, under the conditional law given $\mcl F_n$, the fields $\Phi_2\circ f_\ell + 2\log|f_\ell'| +  L_\ell  $ for $\ell\in\Gamma^n$ are conditionally independent standard embeddings of critical LQG disks with boundary lengths $e^{L_\ell} \nu_{\Phi_2}(\ell)$ (note that this length is $\mcl F_n$-measurable by definition). 
By combining this with~\eqref{eqn-coupling-coord}, the discussion surrounding~\eqref{eqn-coupling-triples}, and the Definition~\ref{def-disk-cle4} of the CLE$_4$-decorated supercritical LQG disk, we see that under the conditional law given $\mcl F_n$, the loop-decorated central charge-$\ccL$ LQG surfaces~\eqref{eqn-cle4-surfaces} are conditionally independent CLE$_4$-decorated LQG disks of central charge $\ccL$ with boundary lengths $e^{L_\ell} \nu_{\Phi_2}(\ell)$.
Since this conditional law depends only on the boundary length $e^{L_\ell} \nu_{\Phi_2}(\ell)$, it follows that we get the same conditional law if instead of conditioning on $\mcl F_n$ we condition only on the supercritical LQG lengths 
\eqbn
\left\{  e^{L_\ell} \nu_{\Phi_2}(\ell) : \ell\in\bigcup_{k=1}^n \Gamma^k \right\} .
\eqen   
\end{proof}

\subsection{Properties of the coupling}
\label{sec-comments}

Throughout this subsection, we assume that we are in the setting of Definition~\ref{def-disk-cle4} for $\ccL\in (1,25)$, so that in particular $(\Phi, \ol\Gamma)$ is a standard embedding of a CLE$_4$-decorated supercritical LQG disk.
We will observe several additional features of the coupling of $\ol\Gamma$ with $\Phi$.  

\subsubsection{Interpolation} 
\label{sec-interpolation}
By Definition~\ref{def-disk-cle4}, our coupling interpolates between independent CLE$_4$ and a critical LQG disk for $\ccL = 25$; and the coupling of a zero-boundary GFF and its CLE$_4$ level lines for $\ccL = 1$.

\subsubsection{Duality between $\ccL$ and $26-\ccL$} 
\label{sec-duality}

Let $\wh\Phi$ be given by the orthogonal linear combination of $\Phi_2$ and $\Psi$ as compared to~\eqref{eqn-coupling-def}, i.e., 
\eqb \label{eqn-dual-disk}
\wh\Phi := \frac{\sqrt{4-Q^2}}{2} \Phi_2  - \frac{Q}{2} \Psi  .
\eqe
Since $1 + 6(\sqrt{4-Q^2})^2 = 26 - \ccL$ and $\Psi\eqD -\Psi$, we see from Definition~\ref{def-supercritical-disk} that $\wh\Phi$ is an embedding of a unit boundary length LQG disk of central charge $26-\ccL$. We have $Q \wh\Phi = \sqrt{4-Q^2} \Phi$ on $\bdy\BB D$, and in the interior of $\BB D$ the fields $\Phi$ and $\wh\Phi$ locally look like two independent GFFs (this is because if we start with a vector of two independent GFFs and apply a $2\times 2$ orthogonal matrix, we again get a vector of two independent GFFs).

The set of level lines $\ol\Gamma$ of $\Psi$ is the same as the set of level lines of $-\Psi$. 
Hence, by Definition~\ref{def-disk-cle4}, the pair $(\wh\Phi , \ol\Gamma)$ is a standard embedding of a CLE$_4$ decorated LQG disk of central charge $26-\ccL$. In other words, we have a coupling $(\Phi,\wh\Phi,\ol\Gamma)$ of standard embeddings of LQG disks of central charges $\ccL$ and $26-\ccL$ with the same CLE$_4$. 
The CLE$_4$ $\ol\Gamma$ is a.s.\ determined by $\Phi$ and $\wh\Phi$ since $\ol\Gamma$ is a.s.\ determined by $\Psi$ and
\eqb 
\Psi = \frac{\sqrt{4-Q^2}}{2} \Phi  -  \frac{Q}{2} \wh\Phi .
\eqe  

Thus, our coupling of $\Phi$ with $\ol\Gamma$ can be viewed as arising from a coupling of $\Phi$ with $\wh\Phi$, wherein the level lines of an appropriate linear combination of $\Phi$ and $\wh\Phi$ divide the LQG surface $(\BB D,\Phi)/{\sim_Q}$ (and also the LQG surface $(\BB D,\wh\Phi)/{\sim_{ \sqrt{4-Q^2 }}}$) into sub-surfaces which are conditionally independent given their boundary lengths. 

Consequently, the coupling of Definition~\ref{def-disk-cle4} is a coupling of two locally independent objects of central charges $\ccL$ and $26-\ccL$, just like the couplings of subcritical or critical LQG with independent SLE (recall~\eqref{eqn-matched}). 

\begin{remark}
The idea of considering LQG with $\ccL \in (1,25)$ decorated by a ``matter field'' which is described by LQG with central charge $26-\ccL$ also appears in the work of Gervais and Roussel on strongly coupled Liouville theory from the perspective of quantum groups~\cite{gervais-solving1,gervais-solving2,gervais-solving3}, see~\cite{gervais-weak-to-strong} for an overview of this work. It is of interest to investigate further how the work of Gervais and Roussel relates to the present paper (see also Problem~\ref{prob-self-dual}). 
\end{remark}

\subsubsection{Inner and outer boundary lengths do not agree} 
\label{sec-gap}

Recall that $\{X_\ell\}_{\ell\in \ol\Gamma}$ are the i.i.d.\ signs associated with the loops, as in~\eqref{eqn-cle-signs}.
Also recall from~\eqref{eqn-level-line-full} that for each $n$th generation loop $\ell\in\Gamma^n$, $\ell^{(n)},\dots,\ell^{(1)}$ denote the CLE$_4$ loops surrounding $\ell$, with $\ell^{(n)} = \ell$. 
 
By the definition of the coupling of $\Psi$ with $\ol\Gamma$ (see in particular~\eqref{eqn-level-line-full}), for each $\ell\in \Gamma^n$, the harmonic extension of the values of the GFF $\Psi$ on $\ell$ is equal to $\pi  \sum_{k=1}^n X_{\ell^{(k)}}$ on the region $U_\ell$ enclosed by $\ell$.
The coupling of $\Psi$ with $\ol\Gamma$ is constructed in~\cite[Section 4.3]{asw-local-sets} via a branching SLE$_4(-2)$ curve coupled with $\Psi$, which traces the loops of $\ol\Gamma$. The harmonic extension of the values of $\Psi$ on each branch of this branching SLE$_4(-2)$ curve differ by $\pm \pi $ on the two sides of the curve. Hence the same is true for each loop $\ell \in \ol\Gamma$.\footnote{The presence of a marked point for each loop (namely, the point where the SLE$_4(-2)$ starts and finishes tracing the loop) does not affect the fact that the harmonic extension of the values of $\Psi$ on the inside and outside of the curve differ by $\pm\pi$: the reason is that changing the value of a function at a single point does not change its harmonic extension.} In other words, the harmonic extension of the values of $\Psi$ on $\ell$ approaches $\pi  \sum_{k=1}^{n-1} X_{\ell^{(k)}}$ as we approach $\ell$ from the outside. 

By the local absolute continuity of the SLE$_4(-2)$ curve with respect to an ordinary SLE$_4$ curve away from the points where it starts and finishes tracing loops, together with~\cite[Theorem 1.2]{hp-welding}, the critical LQG length measure $\nu_{\Phi_2}|_\ell$ can be defined on each loop $\ell \in \Gamma$, and the measure is the same whether we define it using the values of $\Phi_2$ inside of the loop or the values of $\Phi_2$ outside of the loop.

If $\ell\in\Gamma^n$, then the natural LQG length measure on $\ell$ associated with the supercritical LQG disk $(U_\ell , \Phi|_{U_\ell})/{\sim_Q}$ is its supercritical LQG length measure, defined as in~\eqref{eqn-supercritical-length} and Remark~\ref{remark-bdy-coord}. By the definition of $\Phi$ in~\eqref{eqn-coupling-def} and the preceding two paragraphs, this measure is
\eqb \label{eqn-inner-length}
\nu_\Phi^{\op{in}}|_\ell =  \exp\left( \frac{\sqrt{4-Q^2}}{Q}  \pi   \sum_{k=1}^{n } X_{\ell^{(k)}}  \right)  \nu_{\Phi_2}|_\ell .
\eqe 
However, the natural LQG length measure on $\ell$ associated with the LQG surface $(\BB D\setminus \ol U_\ell , \Phi|_{\BB D\setminus \ol U_\ell})$ (defined, e.g., using local absolute continuity with respect to the supercritical LQG disk) is instead
\eqb \label{eqn-outer-length}
\nu_\Phi^{\op{out}}|_\ell = \exp\left( \frac{\sqrt{4-Q^2}}{Q}  \pi   \sum_{k=1}^{n-1} X_{\ell^{(k)}}  \right)  \nu_{\Phi_2}|_\ell .
\eqe  
In other words, we have the following.

\begin{prop} \label{prop-gap}
For each $\ell\in\ol\Gamma$, the supercritical LQG lengths of $\ell$ with respect to $\Phi$ as measured from the inside and outside of $\ell$ differ by a factor of $\exp\left(\frac{\sqrt{4-Q^2}}{Q}   \pi  X_\ell   \right)$.
\end{prop}

This is in contrast to the coupling of a $\gamma$-LQG surface with an independent SLE$_{\kappa=\gamma^2}$-type curve for $\gamma \in (0,2]$, where the LQG lengths as measured from the two sides of the curve agree~\cite{shef-zipper,hp-welding}. 

\subsubsection{Law of the boundary lengths} 
\label{sec-boundary-law}

The joint law of the LQG lengths of the loops of $\ol\Gamma$ with respect to $\Phi$, as measured from the insides of the loops, can be described explicitly as follows. By~\eqref{eqn-inner-length}, for $\ell\in\Gamma^n$ we have
\eqb \label{eqn-inner-length-total}
\nu_\Phi^{\op{in}}(\ell) = \exp\left(  \frac{\sqrt{4-Q^2}}{Q}  \pi   \sum_{k=1}^{n } X_{\ell^{(k)}} \right) \nu_{\Phi_2}(\ell)  ,
\eqe 
where we recall that $\{X_\ell\}_{\ell\in \ol\Gamma}$ are i.i.d.\ Bernoulli$(1/2)$ signs indexed by the loops in $\ol\Gamma$ and $\ell^{(n)},\dots,\ell^{(1)}$ are the loops surrounding $\ell$ (including $\ell$ itself). 

The formula~\eqref{eqn-inner-length-total} describes $\{\nu_\Phi^{\op{in}}(\ell)\}_{\ell\in \ol\Gamma}$ in terms of the critical LQG lengths $\{\nu_{\Phi_2}(\ell)\}_{\ell \in \ol\Gamma}$. By \cite[Theorem 1.2]{ag-critical-cle4}, the joint law of the $\nu_{\Phi_2}$-lengths of the outermost loops of $\ol\Gamma$ can be described in terms of the jumps of a certain $3/2$-stable process (the analogous description in the subcritical case follows from results in~\cite{msw-simple-cle-lqg,ccm-cascade,bbck-growth-frag}, see~\cite[Proposition 5.1]{as-cle-integrability}). The joint law of $\{\nu_{\Phi_2}(\ell)\}_{\ell \in \ol\Gamma}$ can be obtained from this by iterating (see Lemma~\ref{lem-cle4-coupling-critical}). 
The resulting description of the joint law of $\{\nu_\Phi^{\op{in}}(\ell)\}_{\ell\in \ol\Gamma}$ plays a fundamental role in the paper~\cite{bgs-supercritical-crt}. 

Alternatively, as explained in~\cite[Section 5.2.2]{ahps-critical-mating}, the joint law of $\{\nu_{\Phi_2}(\ell)\}_{\ell \in \ol\Gamma}$ is the same as the joint law of the positive jumps of a certain growth fragmentation associated with a Brownian excursion which was introduced in~\cite{ad-growth-fragmentation}.

\subsection{Non-triviality of the coupling} 
\label{sec-non-trivial}
 
In this section we prove that in the setting of Definition~\ref{def-supercritical-disk}, the CLE$_4$ $\ol\Gamma$ is neither independent from nor determined by $\Phi$, and the same is true for the set of outermost loops of $\ol\Gamma$. 
This combined with Theorem~\ref{thm-cle4-coupling-nested} concludes the proof of Theorem~\ref{thm-cle4-coupling}. 
It is easy to see that $\Psi$ is neither independent from nor determined by $\Phi$, but it is not obvious that $\ol\Gamma$ is neither independent from nor determined by $\Phi$: indeed, in the setting of Section~\ref{sec-nested}, $\ol\Gamma$ does not determine the same information as the GFF $\Psi$: rather, to construct $\Psi$ from $\ol\Gamma$ one also needs the signs $\{X_\ell\}_{\ell\in\ol\Gamma}$.

\begin{prop} \label{prop-not-independent}
Let $(\Phi,\ol\Gamma)$ be a standard embedding of a CLE$_4$-decorated LQG disk of unit boundary length (Definition~\ref{def-disk-cle4}). The nested CLE$_4$ $\ol\Gamma$ is not independent from $\Phi$. The same is true for the non-nested CLE$_4$ $\Gamma=\Gamma^1$. 
\end{prop}

Our proof of Proposition~\ref{prop-not-independent} also works for other embeddings of $(\Phi,\ol\Gamma)$, including ones which depend only on the LQG surface $(\BB D , \Phi)/{\sim_Q}$. For example, the same argument would apply if we sampled three points in $\BB D$ from the LQG boundary length measure $\nu_\Phi$ and specified the embedding by the condition that these three points are mapped to $-i, i,$ and $1$.

\begin{proof}[Proof of Proposition~\ref{prop-not-independent}]
It suffices to show that $\Gamma $ is not independent from $\Phi$. 
Basically, this follows since, due to Definition~\ref{def-disk-cle4}, $\Phi$ looks like $Q/2$ times a GFF near each loop of $\Gamma$, instead of like a GFF. Here is one way to make this precise. 

Let $\Phi_2$ and $\Psi$ be as in Definition~\ref{def-disk-cle4}, so that $\Phi = (Q/2)\Phi_2 + (\sqrt{4-Q^2}/2)\Psi$. 
Assume by way of contradiction that $\Gamma$ is independent from $\Phi$. 
Let $\ell\in \Gamma$ be chosen in a manner depending only on $\Gamma$ and let $U_\ell$ be the open region it disconnects from $\infty$. 
For a generalized function $\phi$ on $\BB D$, we write $\op{Harm}_\ell(\phi)$ for the harmonic function on $U_\ell$ with the same boundary data as $\phi$, if such a harmonic function exists. We note that $\op{Harm}_\ell(\Phi)$ and $\op{Harm}_\ell(\Phi_2)$ exist a.s.\ for each $\ell\in\Gamma$ by Theorem~\ref{thm-cle4-coupling-nested} and Lemma~\ref{lem-cle4-coupling-critical}, respectively, and the fact that the harmonic extension of the boundary data of a critical or supercritical LQG disk on $\bdy\BB D$ is well-defined. It follows that the same is true for $\Phi + g$ and $\Phi_2+g$ for any (possibly random) continuous function $g$.

By Definitions~\ref{def-critical-disk} and~\ref{def-supercritical-disk}, we can write 
\eqbn
\Phi = \wt\Phi^0 + g \quad \text{and}\quad
\Phi_2 = \wt\Phi_2^0 + g_2
\eqen 
 where $\wt\Phi^0$ is a zero-boundary GFF on $\BB D$ and $g$ is a continuous function on $\BB D$ (possibly random and $\wt\Phi^0$-dependent), and similarly for $\wt\Phi_2^0$ and $g_2$. In fact, $g_2$ can be taken to be equal to $(2/Q) g$. By our assumption on $(\Phi,\Gamma)$ and since $\Phi_2$ is independent from $\ol\Gamma$ by definition, we can take each of $(\wt\Phi^0,g)$ and $(\wt\Phi_2^0,g_2)$ to also be independent from $\Gamma$ (the two pairs do not need to be jointly independent from $\Gamma$).  
   
By, e.g.,~\cite[Lemma 6.4]{ig4}, for $z \in U_\ell$ the conditional law given $\Gamma$ of $\op{Harm}_\ell(\wt\Phi^0)(z)$ is Gaussian with mean zero and variance $\log \frac{1}{\op{dist}(z,\bdy U_\ell)} + O(1)$, where the $O(1)$ depends on $\ell$ but not on $z$. Since $\op{Harm}_\ell(g)$ is a bounded function, it follows that the conditional law given $\Gamma$ of 
\eqb \label{eqn-harm-conv}
\left( \log \frac{1}{\op{dist}(z,\bdy U_\ell)} \right)^{-1/2} \op{Harm}_\ell(\Phi)(z)
\eqe 
converges to the standard Gaussian distribution as $z\to\bdy U_\ell$. The same is true with $\Phi_2$ in place of $\Phi$, with the same proof. 
 
On the other hand, since $\op{Harm}_\ell(\Psi)$ is equal to the constant function $ \pi  X_\ell$, we have 
\eqbn
\op{Harm}_\ell(\Phi) = \frac{Q}{2} \op{Harm}_\ell(\Phi_2)(z) + \frac{\sqrt{4-Q^2}}{2} \pi  X_\ell .
\eqen
By this and~\eqref{eqn-harm-conv} with $\Phi_2$ in place of $\Phi$, we see that the conditional law given $\Gamma$ of the random variable in~\eqref{eqn-harm-conv} also converges to the law of $Q /2$ times a standard Gaussian as $z\to \bdy U_\ell$. 
This contradicts~\eqref{eqn-harm-conv}.
\end{proof}

We now prove the last remaining assertion of Theorem~\ref{thm-cle4-coupling}, namely that $\ol\Gamma$ is not a.s.\ determined by $\Phi$.

\begin{prop} \label{prop-not-determined}
Let $(\Phi,\ol\Gamma)$ be a standard embedding of a CLE$_4$-decorated LQG disk of unit boundary length (Definition~\ref{def-disk-cle4}). The nested CLE$_4$ $\ol\Gamma$ is not a.s.\ determined by $\Phi$. The same is true for the non-nested CLE$_4$ $\Gamma = \Gamma^1$. 
\end{prop}

Recall that the standard embedding of $(\BB D,\Phi,\ol\Gamma)/{\sim_Q}$ is not a function of the loop-decorated LQG surface $(\BB D,\Phi,\ol\Gamma)/{\sim_Q}$. 
However, as in the case of Proposition~\ref{prop-not-independent}, our proof of Proposition~\ref{prop-not-determined} also works for certain embeddings which depend only on $(\BB D,\Phi)/{\sim_Q}$. From this, we obtain that the loop-decorated LQG surface $(\BB D,\Phi,\ol\Gamma)/{\sim_Q}$ is not a measurable function of the LQG surface $(\BB D,\Phi)/{\sim_Q}$.

\begin{proof}[Proof of Proposition~\ref{prop-not-determined}]
From the construction of the coupling of $\ol\Gamma$ with the zero-boundary GFF $\Psi$ in Section \ref{sec-nested}, we know that a.s.\ $\Psi$ is determined by $\ol\Gamma$ and the signs $\{X_\ell \}_{ \ell\in \ol\Gamma }$. 

By Proposition~\ref{prop-gap}, a.s.\ for each $\ell\in \ol\Gamma$ the ratio of the $\nu_\Phi$ length of $\ell$ as measured from the inside of $\ell$, to the $\nu_\Phi$ length of $\ell$ as measured from the outside of $\ell$, is $\exp\left( \frac{\sqrt{4-Q^2}}{Q} \pi  X_\ell\right)$. Therefore, the signs $\{X_\ell\}_{\ell\in \ol\Gamma}$ are a.s.\ determined by $\Phi$ and $\ol\Gamma$. 

If $\ol\Gamma$ were a.s.\ determined by $\Phi$, then by the preceding paragraph it would follow that also $\{X_\ell\}_{\ell\in \ol\Gamma}$ is a.s.\ determined by $\Phi$, and hence that $\Psi$ is a.s.\ determined by $\Phi$. But, since $\Phi_2$ and $\Psi$ are independent and locally look like GFFs, it is clear from~\eqref{eqn-coupling-def} that $\Psi$ is not a.s.\ determined by $\Phi$. 

Now assume by way of contradiction that $\Gamma = \Gamma^1$ is a.s.\ determined by $\Phi$. 
For $\ell\in\Gamma$, let $\Gamma_\ell$ be the non-nested CLE$_4$ consisting out the outermost loops of $\ol\Gamma|_{U_\ell}$, so that $\Gamma^2 = \bigcup_{\ell\in\Gamma} \Gamma_\ell$.  
By Theorem~\ref{thm-cle4-coupling-nested} and our assumption, we get that for each $\ell\in \Gamma$, the loop-decorated LQG surface $(U_\ell,\Phi|_{U_\ell} , \Gamma_\ell)/{\sim_Q}$ is a.s.\ determined by the LQG surface $(U_\ell,\Phi|_{U_\ell})/{\sim_Q}$. The LQG surfaces $(U_\ell,\Phi|_{U_\ell})/{\sim_Q}$ and their embeddings into $U_\ell$ for $\ell\in\Gamma$ are a.s.\ determined by $\Phi$ and $\Gamma$, hence also by $\Phi$. 

From the previous paragraph, we deduce that the CLE$_4$s $\Gamma_\ell$ for $\ell\in \Gamma$ are a.s.\ determined by $\Phi$, and hence that $\Gamma^2$ is a.s.\ determined by $\Phi$. Applying this same argument inductively shows that $\Gamma^n$ is a.s.\ determined by $\Phi$ for each $n\in\BB N$, hence $\ol\Gamma$ is a.s.\ determined by $\Phi$. But, we previously proved that this is not the case, so we are done. 
\end{proof}

\begin{proof}[Proof of Theorem~\ref{thm-cle4-coupling}]
Let $\Phi$ and $\ol\Gamma$ be coupled as in Definition~\ref{def-disk-cle4}. 
Let $\Gamma = \Gamma^1$ be the non-nested CLE$_4$ consisting of the outermost loops of $\ol\Gamma$.
By Theorem~\ref{thm-cle4-coupling-nested} for $n=1$, this coupling satisfies the Markov property of Theorem~\ref{thm-cle4-coupling}. 
By Propositions~\ref{prop-not-independent} and~\ref{prop-not-determined}, in this coupling $\Gamma$ is neither independent from nor determined by $\Phi$. 
\end{proof}

\section{Additional perspectives}
\label{sec-perspectives}

In Section~\ref{sec-action}, we give a non-rigorous description of the joint law of the central charge-$\ccL$ supercritical LQG disk field $\Phi$ and the ``dual'' central charge-$(26-\ccL)$ supercritical LQG disk field $\wh\Phi$ of Section~\ref{sec-duality} in terms of the Liouville action. 
In Section~\ref{sec-rpm-supercritical}, we exhibit for each $\ccL\in (1,25)$ a family of loop-decorated random planar maps which have a Markov property similar to the one in Theorem~\ref{thm-cle4-coupling-nested} and which we conjecture should converge to CLE$_4$-decorated supercritical LQG.

\subsection{Liouville action interpretation}
\label{sec-action}

Let $\ccL\in (1,25)$ and let $\Phi$ and $\wh\Phi$ be the coupled supercritical LQG disks of central charges $\ccL$ and $26-\ccL$ as in Section~\ref{sec-duality}. In this subsection we derive a heuristic formula for the joint law of $(\Phi,\wh\Phi)$ in terms of a version of the Liouville action (``Proposition''~\ref{prop-supercritical-law}). Nothing in this section is rigorous, and nothing in this section is used later in the paper. A reader who is uninterested in the Liouville action formalism for LQG surfaces can safely skip this section.

We first review the Liouville action in the subcritical and critical cases. 
Fix $\ccL \geq 25$ and let $Q = \sqrt{\frac{\ccL-1}{6}} = \frac{2}{\gamma} + \frac{\gamma}{2} \geq 2$ be the corresponding background charge, where $\gamma \in \BB C$ with $|\gamma|=2$ is the coupling constant. 
Let $g$ be a Riemannian metric on the close unit disk $\ol{\BB D}$ and let $\nabla_g$, $R_g$, $K_g$, $M_g$, and $N_g$, respectively, denote the associated gradient, Ricci scalar curvature, geodesic curvature on $\bdy\BB D$, area measure, and length measure on $\bdy\BB D$. The \textbf{Liouville action with central charge $\ccL$ and background metric $g$} is the functional $\mcl S_{\op{L}}^{\ccL} =\mcl S_{\op{L},g}^{\ccL}$ which takes in a real-valued function $\phi$ on $\ol{\BB D}$ and outputs the quantity
\eqb \label{eqn-liouville-action}
\mcl S_{\op{L}}^{\ccL}(\phi) = \frac{1}{4\pi} \int_{\BB D} (|\nabla_g \phi|^2 + Q R_g\phi) \,dM_g + \frac{1}{2\pi} \int_{\bdy\BB D} (Q K_g \phi + 2\pi \mu e^{(\gamma/2)\phi } ) \, dN_g 
\eqe 
where $\mu > 0$ is a parameter called the \textbf{(boundary) cosmological constant}. Note that we have set the bulk cosmological constant to zero (as in, e.g.,~\cite{rz-boundary-lcft}). This is related to the fact that the LQG disks considered in this paper have fixed boundary length but unconstrained area. 

Heuristically speaking, the (critical or subcritical) LQG disk field with central charge $\ccL \geq 25$ and background metric $g$ is a sample $\Phi$ from the ``probability measure'' which is a normalizing constant times
\eqb \label{eqn-lqg-law}
\exp\left( -\mcl S_{\op{L}}^{\ccL}(\phi) \right) D\phi
\eqe 
where $D\phi$ is the ``uniform measure on functions $\ol{\BB D} \to\BB R$''.

\begin{remark} \label{remark-liouville-field}
For $\ccL> 25$,  \cite{hrv-disk} gave the measure~\eqref{eqn-lqg-law} a mathematically rigorous meaning by specifying that for any nonnegative measurable function $F$ on the space of distributions on $\BB D$ 
\[\int F(\phi) \exp\left( -\mcl S_{\op{L}}^{\ccL}(\phi) \right) D\phi := C_g \op{LF}_{\BB D}(\ccL)[e^{-\mu \nu^\gamma_\Phi(\partial \BB D)}F(\Phi)]\]
where $\op{LF}_{\BB D}(\ccL)$ is the law of the Liouville field from Definition~\ref{def-liouville-field} after coordinate change to $\BB D$,  $\nu_\Phi^\gamma$ denotes the subcritical boundary length measure, and $C_g$ is a constant depending on $g$. On the other hand, $\op{LF}_{\BB D}(\ccL)$ describes the law of the random field obtained by the following procedure: sample a subcritical LQG disk with random boundary length, then embed it in $\BB D$ in a ``uniformly random'' way to obtain a field in $\BB D$ \cite[Theorem 1.2]{ahs-integrability}; see Appendix~\ref{sec-disk-rotate}.
Hence a random field sampled from the law~\eqref{eqn-lqg-law} is not exactly the same as the (subcritical analog of the) LQG disks considered in this paper (e.g., because the LQG disks in this paper have fixed boundary length), but rather is a minor variant. Analogous statements can be shown for $\ccL = 25$.
\end{remark}
 
It is not clear a priori how to define the Liouville action~\eqref{eqn-liouville-action} when $\ccL \in (1,25)$ since in this case $\gamma $ is complex, so directly extending the relation~\eqref{eqn-liouville-action} would give us a complex action functional, which cannot be used to define a probability measure. As we will now see, our notion of the supercritical LQG disk can be described in terms of a certain modified version of the Liouville action where we force the term inside the exponential to be real (``Proposition'' \ref{prop-supercritical-law}). 

To this end, let $\ccL \in (1,25)$ and let $Q$ be the corresponding background charge. Let $\Phi_2$ be sampled from the measure~\eqref{eqn-lqg-law}, so that $\Phi_2$ describes a critical LQG disk. Let $\Psi$ be an independent zero-boundary GFF on $\BB D$; heuristically, its law is a normalizing constant times the formal measure
\eqb \label{eqn-gff-law}
\exp\left( -\int_{ \BB D}  |\nabla_g \psi|^2   \, dM_g \right) \BB 1\{\psi|_{\bdy\BB D} = 0\}D\psi 
= \exp\left( -\int_{ \BB D}  |\nabla  \psi(z)|^2   \, d^2 z \right) \BB 1\{\psi|_{\bdy\BB D} = 0\} D\psi
\eqe 
where again $D\psi$ denotes the ``uniform measure on functions $\ol{\BB D}\to\BB R$''. 

Define the fields
\eqb \label{eqn-action-rotate}
\Phi = \frac{Q}{2} \Phi_2 + \frac{\sqrt{4-Q^2}}{2} \Psi \quad \text{and} \quad
\wh\Phi = \frac{\sqrt{4-Q^2}}{2}  \Phi_2 - \frac{Q}{2}   \Psi  
\eqe  
which by Definition~\ref{def-supercritical-disk} are (minor variants of; c.f.\ Remark~\ref{remark-liouville-field}) supercritical LQG disks of central charge $\ccL$ and $26-\ccL$, respectively (see also the discussion in Section~\ref{sec-duality}).
We will derive the following heuristic expression for the joint law of $(\Phi,\wh\Phi)$. 

\begin{propq} \label{prop-supercritical-law}
As in~\eqref{eqn-lqg-parameters}, let
\eqb \label{eqn-supercritical-gamma}
\gamma = Q + i\sqrt{4-Q^2} \quad \text{so that} \quad Q = \frac{2}{\gamma} + \frac{\gamma}{2} . 
\eqe 
For each $\ccL \in (1,25)$, the joint law of $(\Phi,\wh\Phi)$ is given by a normalizing constant times
\allb \label{eqn-supercritical-law}
 \exp\left( -  \wt{\mcl S}_{\op{L}}^{\ccL , 26-\ccL}(\phi,\wh\phi) \right)   \BB 1\{ \sqrt{4-Q^2} \phi|_{\bdy\BB D} = Q\wh\phi|_{\bdy\BB D}\} D\phi\,D\wh\phi  
\alle
where $D\phi \, D\wh\phi$ is the ``uniform measure on pairs of functions on $\ol{\BB D}$'' and
\allb \label{eqn-supercritical-action}
\wt{\mcl S}_{\op{L}}^{\ccL , 26-\ccL}(\phi,\wh\phi)
&= \frac{1}{4\pi}  \int_{\BB D} (| \nabla_g \phi |^2 +  Q R_g  \phi)  \,dM_g   + \frac{1}{4\pi}  \int_{\BB D} (| \nabla_g \wh\phi |^2 +  \sqrt{4-Q^2} R_g  \wh\phi) \,dM_g   \notag\\
&\qquad + \frac{1}{2\pi} \int_{\bdy\BB D} (Q K_g \phi + \sqrt{4-Q^2} K_g \wh\phi + 2\pi \mu  e^{(\gamma/2) (\phi - i\wh\phi) } )\, dN_g . 
\alle  
\end{propq}

The action functional $\wt{\mcl S}_{\op{L}}^{\ccL , 26-\ccL}(\phi,\wh\phi)$ of~\eqref{eqn-supercritical-action} is a modified version of $\wt{\mcl S}_{\op{L}}^{\ccL  }(\phi ) + \wt{\mcl S}_{\op{L}}^{ 26-\ccL}( \wh\phi)$ where we have $e^{(\gamma/2)(\phi - i \wh\phi) }$ instead of $e^{(\gamma/2)\phi}  + e^{(\wh\gamma/2)\wh\phi}$ in the boundary term.  
The expression $\phi - i \wh \phi$ is reasonable from the perspective of conformal field theory. It is analogous to the complex field $\phi_\mathrm{L} - i \phi_\mathrm{M}$ where $\phi_\mathrm{L}$ is a Liouville field with coupling constant $\gamma^0 \in (0,2]$ and central charge $\cc_\mathrm{L}^0$, and  $\phi_\mathrm{M}$ is a matter field with central charge $26 - \cc_\mathrm{L}^0$, see \cite{witten-ground-ring} for $\cc_\mathrm{L}^0 =	 25$ or \cite[Sections 3--4]{kostov-ground-ring} for  $\cc_\mathrm{L}^0 > 25$ for examples where this complex field arises in the physics literature.
The truncation $\BB 1\{\sqrt{4-Q^2} \phi|_{\bdy\BB D} = Q\wh\phi|_{\bdy\BB D}\}$ is precisely what is needed in order to make $(\gamma/2)( \phi  - i \wh\phi)$ real on $\bdy\BB D$, which in turn ensures that $\wt{\mcl S}_{\op{L}}^{\ccL , 26-\ccL}(\phi,\wh\phi)$ is real.
We note that $(\gamma/2)(\phi - i\wh\phi) = (\wh\gamma/2)(\wh\phi - i \phi)$ where $\wh\gamma = \sqrt{4-Q^2} - i Q$, which illustrates the symmetric roles played by $\phi$ and $\wh\phi$.

\begin{proof}[Heuristic derivation of ``Proposition'' \ref{prop-supercritical-law}]
By~\eqref{eqn-liouville-action} and~\eqref{eqn-gff-law}, the joint law of $(\Phi_2,\Psi)$ is proportional to
\eqb \label{eqn-joint-law}
\exp\left( -\mcl S_{\op{L}}^{\ccL=25}(\phi_2) -\int_{ \BB D}  |\nabla_g \psi|^2   \, dM_g \right)  \BB 1\{\psi|_{\bdy\BB D} = 0\} D\phi_2 \,D\psi .
\eqe 
To convert this to an expression for the joint law of $(\Phi,\wh\Phi)$ defined in~\eqref{eqn-action-rotate}, we make the change of variables
\eqbn
\phi_2 = \frac{Q}{2} \phi + \frac{\sqrt{4-Q^2}}{2} \wh\phi ,\quad \psi = \frac{\sqrt{4-Q^2}}{2} \phi - \frac{Q}{2} \wh\phi .
\eqen
This is an orthogonal transformation of the pair of functions $(\phi,\wh\phi)$, so $D\phi_2\,D\psi$ becomes the ``uniform measure'' $D\phi\,D\wh\phi$ on pairs of functions $(\phi,\wh\phi)$ on $\ol{\BB D}$. By direct substitution, we also have
\allb
&\BB 1\{\psi|_{\bdy\BB D} = 0\} = \BB 1\{ \sqrt{4-Q^2} \phi|_{\bdy\BB D} = Q\wh\phi|_{\bdy\BB D}\} \notag\\
&|\nabla_g \phi_2|^2 + |\nabla_g \psi|^2 = |\nabla_g \phi |^2 + |\nabla_g \wh\phi|^2   \notag\\
&2R_g \phi_2 =  Q R_g \phi + \sqrt{4-Q^2} R_g \wh\phi  \quad \text{and} \quad
2K_g \phi_2 =  Q K_g \phi + \sqrt{4-Q^2} K_g \wh\phi \notag\\
&e^{\phi_2}|_{\bdy\BB D} 
= \exp\left(  \frac{Q}{2} \phi + \frac{\sqrt{4-Q^2}}{2} \wh\phi \right) |_{\bdy\BB D}      
= \exp\left( \frac{\gamma}{2} (\phi - i \wh\phi) \right) |_{\bdy\BB D}  .
\alle 
Note that the imaginary part of $\frac{\gamma}{2} (\phi - i \wh\phi)$ is zero due to the relation $\sqrt{4-Q^2} \phi|_{\bdy\BB D} = Q\wh\phi|_{\bdy\BB D}$. 
Plugging the above equations into~\eqref{eqn-liouville-action} (with $2$ instead of $Q$ and $\phi_2$ instead of $\phi$) and then into~\eqref{eqn-joint-law}, we obtain~\eqref{eqn-supercritical-law}.
\end{proof}

\subsection{A random planar map model for supercritical LQG} 
\label{sec-rpm-supercritical}
 
There are various loop-decorated random planar maps which exhibit exact discrete analogs of the properties of the coupling of the supercritical LQG disk and nested CLE$_4$ in Definition~\ref{def-disk-cle4} (see also Theorem~\ref{thm-cle4-coupling-nested} and Section~\ref{sec-comments}). This leads to precise scaling limit conjectures for such loop-decorated random planar maps toward supercritical LQG disks coupled to CLE$_4$. In this subsection we will explain one such conjecture in detail. This represents the first scaling limit conjecture for a combinatorially natural random planar map model toward supercritical LQG.\footnote{The paper~\cite{ghpr-central-charge} considers a random planar map constructed using the GFF which is expected to converge to supercritical LQG, but this random planar map is not combinatorially natural since its definition is in terms of continuum objects.}

\begin{remark} \label{remark-rpm}
A central difficulty in describing random planar maps which converge to supercritical LQG is that such random planar maps should be infinite, with infinitely many ``ends''. Indeed, this is because a supercritical LQG surface, equipped with its supercritical LQG metric, has infinitely many ``singular points'' which lie at infinite distance from every other point~\cite{dg-uniqueness} (see~\cite{ghpr-central-charge} for some related discussion). Consequently, one cannot simply define random planar maps which converge to supercritical LQG in terms of a Radon-Nikodym derivative with respect to the uniform measure on some finite set. The approach we take in this subsection is to instead define our random planar maps recursively via a multi-type branching process.

We expect that there are also loop-decorated random planar maps in the same universality class as the model considered in this subsection which can be obtained, roughly speaking, as follows. Start from a finite random planar map decorated by an appropriate statistical physics model. Then (using ideas from the theory of branching processes) find a canonical random planar map which is infinite with high probability and which has the property that conditioning the map to be finite gives the original finite random planar map. For example, it may be possible to do this for planar maps decorated by versions of the $O(n)$ loop model or the six vertex model. We plan to explore this further in future work. 

In another direction, the recent work~\cite{cps-random-surface-ym} introduces some random planar maps related to lattice Yang-Mills theory which may have connections to LQG with $\ccL \in (1,25)$ (see the open problems in~\cite{cps-random-surface-ym}). It would be of interest to explore these possible connections further.
\end{remark}

\subsubsection{Triangulations decorated by the fully-packed $O(2)$ loop model}

Our random planar map model for supercritical LQG will be a modified version of triangulations of the disk decorated by the fully packed $O(2)$ loop model. We first review the definition of such decorated triangulations. 

For $k\in\BB N$, let $\mcl M_k^{\op{FPL}}$ be the set of pairs $(M,L)$ where $M$ is a triangulation with (not necessarily simple) boundary of perimeter $k$ and $L$ is a \textbf{fully packed loop configuration} on the dual map $M^*$, i.e., a collection of simple cycles in $M^*$ with the property that each vertex of $M^*$ (i.e., each face of $M$) is visited by exactly one loop. See Figure~\ref{fig-map-gasket}, panel (4) for an illustration of an element of $\mcl M_k^{\op{FPL}}$. For a parameter $\beta > 0$, we define a probability measure on $\mcl M_k^{\op{FPL}}$ by
\eqb \label{eqn-planar-map-O(2)}
\BB P\left[ (M,L) \right] = Z_k^{-1} e^{-\beta \# \mcl F(M)} 2^{\# L} ,
\eqe 
where $\mcl F(M)$ is the set of faces of $M$ and $Z_k$ is a normalizing constant to make the total mass one. 

If $\beta$ is large enough, then~\eqref{eqn-planar-map-O(2)} is well-defined in the sense that the sum $Z_k$ of $e^{-\beta \# \mcl F(M)} 2^{\# L}$ over all elements of $\mcl M_k^{\op{FPL}}$ is finite. Let $\beta_c$ be the smallest value of $\beta$ for which this is the case. It is expected that for $\beta = \beta_c$, we have $Z_k < \infty$ and $Z_k$ decays like $k^{-2}$ as $k\to\infty$ (see, e.g.,~\cite[Section 8]{legall-miermont-large-faces} and the references therein). We henceforth assume that $\beta = \beta_c$. 

Note that under the probability measure~\eqref{eqn-planar-map-O(2)}, the conditional law of $L$ given $M$ is that of the fully packed $O(2)$ loop model on $M^*$. We have the following well-known folklore conjecture for the scaling limit of random planar maps sampled from~\eqref{eqn-planar-map-O(2)}. 

\begin{conj} \label{conj-O(2)}
For $k\in\BB N$, let $(M_k,L_k)$ be sampled from~\eqref{eqn-planar-map-O(2)} with $\beta=\beta_c$. Then $(M_k,L_k)$ converges in law under an appropriate scaling as $k\to\infty$ to a critical ($\gamma=2$) LQG disk with unit boundary length together with an independent CLE$_4$.
\end{conj}

Possibly topologies of convergence in Conjecture~\ref{conj-O(2)} include versions of the Gromov-Hausdorff convergence for metric spaces decorated by loops (see, e.g.,~\cite[Section 6.1]{ghs-metric-peano}) and convergence of the loop-decorated random planar map under an appropriate embedding into $\BB D$ (e.g., circle packing or Tutte embedding). 

\subsubsection{Gasket decomposition}
\label{sec-gasket}

Our goal is to modify the pair $(M_k,L_k)$ to get a family of random planar maps which should converge in law to supercritical LQG coupled with CLE$_4$.  

To this end, we recall the \textbf{gasket decomposition} of $(M,L)$~\cite{bbg-recursive-approach,bbg-bending}. For $(M,L) \in \mcl M_k^{\op{FPL}}$, the \textbf{gasket} of $(M,L)$ is the planar map whose vertex (resp.\ edge) set consists of all vertices (resp.\ edges) of $M$ which can be reached by a path of edges started from $\bdy M$ which does not cross any loop of $L$, and whose face set consists of the ``holes'' enclosed by these edges. Each face of $G$ corresponds to one of the outermost loops of $L$ (i.e., loops which are not disconnected from $\bdy M$ by any other loops). The planar map $G$ is a discrete analog of the CLE$_4$ gasket, i.e., the set of points in the disk which are not surrounded by any CLE$_4$ loop. 

For $k \geq 1$ and $k'\geq 0$, a \textbf{ring} with outer boundary length $k$ and inner boundary length $k'$ is a planar map consisting of a cyclically ordered set of triangles, each of which shares one edge with each of its neighbors, plus two additional marked faces with perimeters $k$ and $k'$, respectively, which we call the \textbf{outer face} and the \textbf{inner face} (by convention, a face of perimeter 0 consists of a single vertex). For $(M,L) \in \mcl M_k^{\op{FPL}}$, each loop $\ell\in L$ gives rise to a ring by looking at the triangles visited by $\ell$ and forgetting any identifications between pairs of edges of these triangles which are not crossed by $\ell$. 

\begin{figure}[ht!]
\begin{center}
\includegraphics[width=0.75\textwidth]{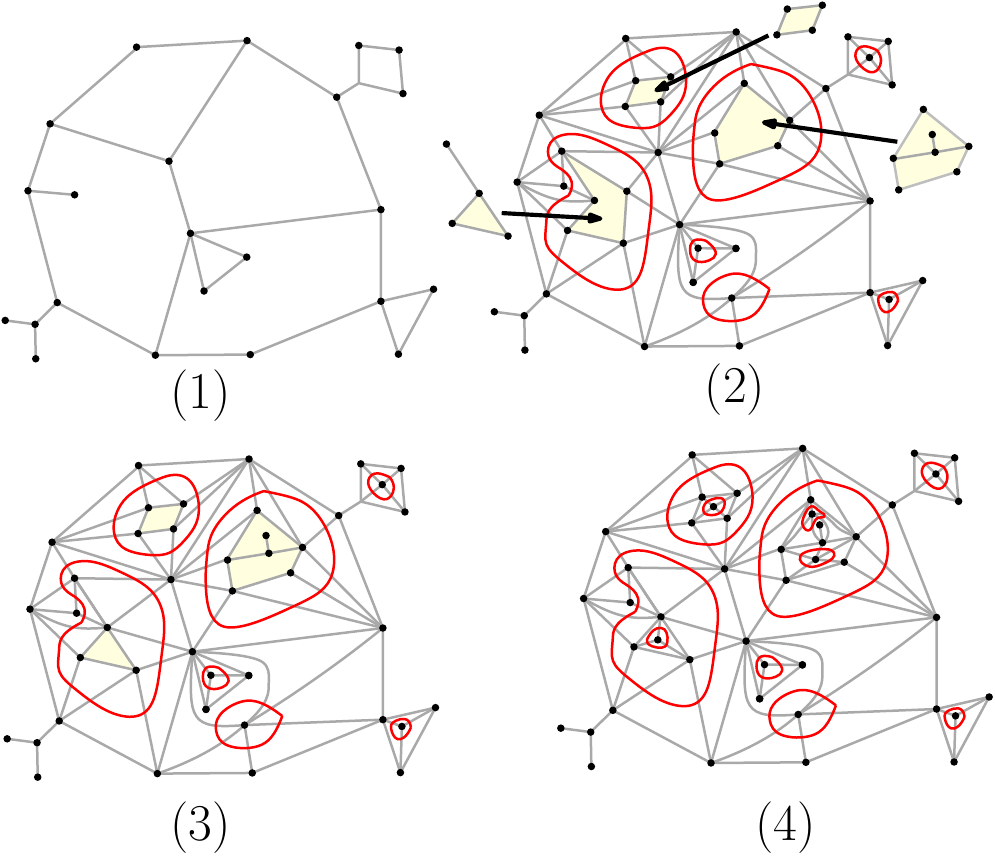}  
\caption{\label{fig-map-gasket} The iterative procedure for sampling a planar map decorated by the fully packed $O(2)$ loop model. \textbf{(1)} The initial gasket $G$. \textbf{(2)} For each $f\in\mcl F(G)$, we sample a ring $R$ decorated by a loop $\ell$ and identify the outer boundary of the ring with the inner boundary of $f$ (the inner face of each ring is shown in light yellow). Then, we sample another gasket with outer boundary length equal to the inner boundary length of $\ell$. \textbf{(3)} We identify the outer boundary of each second-generation gasket with the inner boundary of its corresponding ring. Note that two edges on the inner boundary of one of the rings get identified to the same multiplicity-two edge on the boundary of one of the gaskets. \textbf{(4)} We iterate the procedure inside each face of the second-generation gaskets. In this case, each of the second-generation rings has zero inner boundary length, so the process terminates and the resulting loop-decorated planar map is a sample from~\eqref{eqn-planar-map-O(2)}. 
}
\end{center}
\vspace{-3ex}
\end{figure}

As explained in~\cite{bbg-recursive-approach,bbg-bending}, for $k\in\BB N$ a sample $(M,L)$ from the probability measure~\eqref{eqn-planar-map-O(2)} can be generated via the following recursive procedure. See Figure~\ref{fig-map-gasket} for an illustration.
\begin{enumerate}[$(i)$]
\item Sample the gasket $G$ of $(M,L)$ from its marginal law.
\item \label{step-ring} For each hole $f \in \mcl F(G)$ of perimeter $\op{perim}(f)$, sample a ring $R$ with outer boundary length $\op{perim}(f)$ and random inner boundary length from an appropriate probability measure depending only on $\op{perim}(f)$. 
\item Identify the edges of the outer boundary of $R$ with the edges of the boundary of $f$ in cyclic order. Let $\ell $ be the loop which traverses the triangles of $R$ in cyclic order.
\item Let $k'$ be the inner boundary length of $R$. Conditional on $R$, let $G_f$ be sampled from the gasket of a loop-decorated map with the law~\eqref{eqn-planar-map-O(2)} with $k'$ in place of $k$ (if $k' = 0$, we take $G_f=\emptyset$). Identify the edges of the outer boundary of $G_f$ with the edges of the inner boundary of $R$ in cyclic order.
\item After performing the above procedure for each face $f\in \mcl F(G)$, we iterate the procedure in each face $f'\in \bigcup_{f\in \mcl F(G)} G_f$. Almost surely, this procedure will eventually terminate, in the sense that after finitely many iterations the inner boundary lengths of all of the rings will be zero. At this point the resulting planar map decorated by all of the loops $\ell$ in the rings $R$ will have the law~\eqref{eqn-planar-map-O(2)}. 
\end{enumerate}
The above procedure gives us a description of a sample from~\eqref{eqn-planar-map-O(2)} in terms of a multi-type branching process, where the types are the perimeters of the holes (faces of the gasket) at each stage of the iteration. This procedure can be seen as a discrete analog of the Markov property for the CLE$_4$-decorated critical LQG disk (i.e., Theorem~\ref{thm-cle4-coupling-nested} for $\ccL = 25$).

\subsubsection{Modifying the gasket decomposition}
\label{sec-gasket-modify}

We now want to modify the recursive procedure of Section~\ref{sec-gasket} to get a loop-decorated random planar map which should converge to CLE$_4$-decorated supercritical LQG, for a given $\ccL\in (1,25)$. We first use our continuum theory to determine what modifications we should make.

In the setting of Definition~\ref{def-disk-cle4}, let 
\eqbn
\op{gask}(\Gamma^1) = \ol{\bigcup_{\ell\in\Gamma^1} \ell} 
\eqen
be the gasket of the non-nested CLE$_4$ $\Gamma^1$ consisting of the outermost loops of $\ol\Gamma$. 
Then heuristically speaking we have
\eqbn
\Phi|_{ \op{gask}(\Gamma^1)} = \frac{Q}{2} \Phi_2|_{\op{gask}(\Gamma^1)} ,
\eqen
where $\Phi$ and $\Phi_2$ are the supercritical and critical LQG disk fields from Definition~\ref{def-disk-cle4}, respectively.
Indeed, we have $\Phi = (Q/2)\Phi_2 + (\sqrt{4-Q^2}/2)\Psi$ and $\op{gask}(\Gamma^1)$ is a zero level set of $\Psi$, so heuristically $\Psi|_{\op{gask}(\Gamma^1)} = 0$.\footnote{The height of the field $\Psi$ jumps by $\pm\pi $ when we cross each loop $\ell \in \Gamma^1$, but the values of $\Psi$ on $\ell$ are still zero when viewed from the exterior of the loop. 
}
Since our linear combination of $\Phi_2$ and $\Psi$ behaves nicely under LQG coordinate change (see the proof of Lemma~\ref{lem-bdy-coord}), this means that critical LQG objects associated with $\Phi_2|_{\op{gask}(\Gamma^1)}$ can also be viewed as supercritical LQG objects associated with $\Phi|_{\op{gask}(\Gamma^1)}$, just like in the case of the critical and supercritical LQG length measures~\eqref{eqn-supercritical-length}.\footnote{For example, the supercritical LQG measure on $\op{gask}(\Gamma^1)$ with respect to $\Phi$ should be the same as the critical LQG measure on $\op{gask}(\Gamma^1)$ with respect to $\Phi_2$. A rigorous construction of this measure has not yet been written down, but it should arise, e.g., as a Gaussian multiplicative chaos with respect to the canonical Euclidean measure on $\op{gask}(\Gamma)$~\cite[Proposition 4.5]{ms-cle-measure}. As another example, the supercritical LQG intrinsic metric on $\op{gask}(\Gamma^1)$ with respect to $\Phi$ should be the same as the critical LQG intrinsic metric on $\op{gask}(\Gamma^1)$ with respect to $\Phi_2$. This measure has also not been rigorously constructed, but see~\cite[Section 1.3]{msw-simple-cle-lqg}.}
This suggests that the gasket of a loop-decorated random planar map which converges to the CLE$_4$-decorated supercritical LQG disk as in Definition~\ref{def-disk-cle4} should be in the same universality class as the gasket of the fully packed $O(2)$-loop model decorated triangulation. 

For a critical LQG disk decorated by an independent CLE$_4$, the LQG length of each of the CLE$_4$ loops as measured from the inside and outside of the loop are the same (this follows, e.g., from~\cite[Theorem 1.2]{hp-welding} and a local absolute continuity argument). Hence, Conjecture~\ref{conj-O(2)} suggests that when $\op{perim}(f)$ is large, the measure on rings in Step~\eqref{step-ring} above should be supported on rings whose inner and outer boundary lengths differ by a quantity which is much smaller than $\op{perim}(f) $. 

In contrast, by Proposition~\ref{prop-gap}, in the setting of Definition~\ref{def-disk-cle4}, for each CLE$_4$ loop $\ell \in \ol\Gamma$, the supercritical LQG length of $\ell$ as measured from the outside of the loop and the inside of the loop differ by a factor of $  \exp\left( \pm \frac{\sqrt{4-Q^2}}{Q}  \pi\right) $, where the signs are chosen by independent fair coin flips for each CLE$_4$ loop.

In light of the above discussion, to construct a loop-decorated planar map which should converge to the CLE$_4$ decorated supercritical LQG disk, one should modify the iterative procedure of Section~\ref{sec-gasket} so that we still add the gasket of a fully packed $O(2)$-loop model decorated triangulation at each step, but we force a discrepancy between the inner and outer boundary lengths of the rings.

To this end, fix $\ccL \in (1,25)$ and let 
\eqb \label{eqn-ring-ratio}
s := \exp\left( \frac{\sqrt{4-Q^2}}{Q} \pi \right) , \quad \text{where} \quad Q = \sqrt{\frac{\ccL-1}{6}}   .
\eqe
We perform exactly the same iterative procedure as in Section~\ref{sec-gasket}, except that step~\eqref{step-ring} is replaced by the following procedure.
\begin{enumerate}
\item[$(ii')$] \label{step-ring'} For each hole $f \in \mcl F(G)$ of perimeter $\op{perim}(f)$, independently flip a fair coin. If the coin comes up heads, sample a uniformly random ring $R$ with outer boundary length $\op{perim}(f)$ and inner boundary length $\lfloor s \op{perim}(f) \rfloor$. If the coin comes up tails, instead sample a uniformly random ring $R$ with outer boundary length $\op{perim}(f)$ and inner boundary length $\lfloor s^{-1} \op{perim}(f) \rfloor$.  
\end{enumerate}
\newcommand{\refRing}{{(\hyperref[step-ring']{$ii'$})}}
The other steps in the above procedure remain unchanged. 

Due to Conjecture~\ref{conj-O(2)}, the discussion preceding~\refRing, and the Markov property of the CLE$_4$-decorated supercritical LQG disk given in Theorem~\ref{thm-cle4-coupling-nested}, we expect that the loop-decorated planar map obtained from the above modified iteration procedure should converge to the CLE$_4$ decorated supercritical LQG disk. 

Since supercritical LQG surfaces have ``spikes'' of infinite diameter (singular points) when equipped with the corresponding supercritical LQG metric~\cite{dg-uniqueness}, we expect that the planar maps obtained from the above procedure typically have infinite graph distance diameter, i.e., the associated multi-type branching process should be supercritical. This leads to the following two conjectures.

\begin{conj} \label{conj-supercritical}
For any $s>1$, if we perform the modified iteration procedure as in~\refRing, it holds with probability tending to one as the initial boundary length $k$ tends to $\infty$ that the procedure does \emph{not} terminate after only finitely many steps. Hence, with high probability when $k$ is large, after iterating countably many times we obtain an infinite triangulation $M_k'$ with boundary of perimeter $k$ decorated by an (infinite) fully packed loop configuration $L_k'$ on its dual. 
\end{conj}

\begin{conj} \label{conj-supercritical-conv}
For $k\in\BB N$, let $(M_k',L_k')$ be sampled from the modified iteration procedure as in Conjecture~\ref{conj-supercritical}. Then $(M_k',L_k')$ converges under an appropriate scaling limit as $k\to\infty$ to a CLE$_4$-decorated supercritical LQG disk as in Definition~\ref{def-disk-cle4}, with the central charge parameter $\ccL \in (1,25)$ depending on $s$ as in~\eqref{eqn-ring-ratio}.  
\end{conj}

Regarding the topology of convergence in Conjecture~\ref{conj-supercritical-conv}, we expect that the map $M_k'$ admits an embedding into the disk $\BB D$ wherein the vertices accumulate at infinitely many interior points of $\BB D$, and under which $L_k'$ converges to CLE$_4$ and the graph distance on $M_k'$ (appropriately re-scaled) converges to the supercritical LQG metric. Constructing such an embedding is already an interesting problem (see Problem~\ref{prob-rpm-embedding}). Alternatively, one could attempt to use a variant of Gromov-Hausdorff convergence, but it is not clear what variant is appropriate since the supercritical LQG metric is not locally compact~\cite[Proposition 1.14]{pfeffer-supercritical-lqg}.

\section{Open problems}
\label{sec-open-problems}

We discuss some additional open problems besides the ones already presented in Section~\ref{sec-rpm-supercritical} (see Conjectures~\ref{conj-supercritical} and~\ref{conj-supercritical-conv}). 
In Section~\ref{sec-supercritical-disk}, we gave a definition of canonical supercritical LQG surfaces with the disk topology. It is of interest to extend this construction to LQG surfaces with other topologies. 

\begin{prob} \label{prob-other-supercritical}
What is the right notion of supercritical LQG surfaces with other topologies, e.g., the sphere, the annulus, the half-plane, or the torus? What about supercritical LQG surfaces with various types of bulk or boundary insertions (log singularities)? Do such surfaces have nice relationships with SLE and/or random planar maps?
\end{prob}

See Remark~\ref{remark-free-boundary} for some additional discussion related to Problem~\ref{prob-other-supercritical}.
One special case of Problem~\ref{prob-other-supercritical} which can already be addressed via the methods of this paper is as follows.
By~\cite{ss-contour}, one view an SLE$_4$ curve from 0 to $\infty$ in the upper half-plane $\BB H$ as a level line of the field $\Psi$ on $\BB H$ which is a zero-boundary GFF plus $ \op{arg}(\cdot)$. By combining this with the relationship between SLE$_4$ and the $\gamma=2,\alpha=1$ LQG wedge~\cite{hp-welding} and using a construction similar to the one of Definition~\ref{def-disk-cle4}, one gets a coupling of an SLE$_4$ curve with a certain supercritical LQG surface parametrized by $\BB H$. In this coupling, the supercritical LQG surfaces parametrized by the regions to the left and right of the SLE$_4$ curve are independent and each has similar local behavior to the supercritical LQG disk (Definition~\ref{def-supercritical-disk}). The original supercritical LQG surface has similar local properties to the supercritical LQG disk away from the origin, but has a singularity of the form $ \frac{Q}{2} \log|\cdot| +  \frac{\sqrt{4-Q^2}}{2}  \op{arg}(\cdot)$ at the origin.

\begin{prob} \label{prob-single-sle4}
What else can be said about supercritical LQG surfaces with singularities of the form $\alpha \log|\cdot| + \beta \op{arg}(\cdot)$ and their relationship to SLE$_4$? 
Is there a corresponding random planar map story? 
\end{prob} 

There are other local sets for the zero-boundary GFF $\Psi$, besides just its CLE$_4$ level lines, with the property that the boundary data for $\Psi$ is constant on each complementary connected component of the set. Such sets include the \textbf{first passage sets} and \textbf{two-valued local sets} of $\Psi$ studied in~\cite{asw-local-sets,aru-sepulveda-2valued}. The argument of Theorem~\ref{thm-cle4-coupling-nested} immediately extends to give similar coupling statements for these sets with the supercritical LQG disk, provided one first has an appropriate Markov property when we look at the set together with an independent critical LQG disk. This motivates the following problem.  

\begin{prob} \label{prob-critical-disk}
Suppose that $A$ is one of the special local sets of the GFF $\Psi$ mentioned above, and let $\Phi_2$ be an embedding of a critical LQG disk, sampled independently from $A$. Show that the critical LQG surfaces parametrized by the complementary connected components of $A$ are conditionally independent critical LQG disks given their boundary lengths. Use this to get a coupling of $A$ with the supercritical LQG disk. Does this coupling have any interesting features which are not seen in the coupling with CLE$_4$ considered in this paper? 
\end{prob}

There are various ways of embedding planar maps into the plane under which certain random planar maps are expected (and in some cases proven) to converge to subcritical or critical LQG. Examples include circle packing (see the overview in~\cite{stephenson-circle-packing}), Tutte embedding (see, e.g.,~\cite{gms-tutte}), Smith diagrams~\cite{brooks-dissection} (see also~\cite{bgs-smith}), Cardy embedding~\cite{hs-cardy-embedding}, and Riemann uniformization of the surface obtained by identifying the faces with regular polygons of unit side length.

\begin{prob} \label{prob-rpm-embedding}
Are there variants of any of the above embeddings which are defined for an infinite planar map with boundary with infinitely many ends, e.g., a typical realization of one of the random planar maps defined in Section~\ref{sec-gasket-modify}? What can be said about these embeddings? 
\end{prob}

The motivation for Problem~\ref{prob-rpm-embedding} is, of course, to find an embedding under which random planar maps (possibly decorated by loops) converge to the supercritical LQG disk (possibly decorated by CLE$_4$). The embeddings we are looking for should have the property that there is an infinite, totally disconnected set of points in the disk where the vertices of the planar map accumulate. 
Embedded random planar maps which converge to supercritical LQG should have similar properties to the random dyadic tilings of the plane constructed from the GFF in~\cite{ghpr-central-charge}.

\medskip
\noindent\textit{Update:} A version of the Tutte embedding for planar maps with boundary having infinitely many ends is constructed in the forthcoming paper~\cite{gs-reflected-rw}. It is still open to define versions of other embeddings for such maps, such as circle packing and Riemann uniformization.
\medskip

As a first step toward proving the random planar map convergence of Conjecture~\ref{conj-supercritical-conv}, one could try to prove the following.

\begin{prob} \label{prob-rpm-loops}
Show that in the setting of Section~\ref{sec-gasket-modify}, the joint law of the (appropriately re-scaled) lengths of the loops $L_k'$ converges to the joint law of the (inner) supercritical LQG lengths of the CLE$_4$ loops on a CLE$_4$-decorated supercritical LQG disk, as described in Section~\ref{sec-boundary-law}.
\end{prob}

The analog of Problem~\ref{prob-rpm-loops} for the $O(n)$ loop model on a random quadrangulation, with $n \in (0,2)$, has already been solved in~\cite{bbg-recursive-approach,msw-simple-cle-lqg,ccm-cascade,bbck-growth-frag}.
Due to the relationship between the map of Section~\ref{sec-gasket-modify} and the $O(2)$ model on a random triangulation, it appears that the main step needed to solve Problem~\ref{prob-rpm-loops} is to generalize the results of~\cite{bbg-recursive-approach,ccm-cascade} to the case of $O(2)$-decorated triangulations. We remark that the paper~\cite{bcm-cauchy-process} proves a number of properties of planar maps whose face degree distribution is in the domain of attraction of a Cauchy process; the gasket of an $O(2)$-decorated triangulation is believed, but not proven, to have this property. 

\medskip
\noindent\textit{Update:} It is proven independently in~\cite{kammerer-O2-gasket,ash-volume} that the gaskets of random bipartite planar maps decorated by the critical rigid $O(2)$ loop model are discrete $3/2$-stable maps. This result, combined with~\cite[Theorem 1.2]{bgs-supercritical-crt}, leads to a solution to the variant of Problem~\ref{prob-rpm-loops} with these gaskets used in place of the gaskets of $O(2)$-decorated triangulations in the construction of Section~\ref{sec-rpm-supercritical}. 
\medskip

The central objects of study in Liouville conformal field theory are the \textbf{Liouville correlation functions}, which are regularized versions of expectations of the form
\eqb \label{eqn-corr-function}
\BB E\left[  e^{-\mu \nu_\Phi(\bdy U)}  \prod_{j=1}^n e^{\alpha_j \Phi(z_j)}   \right] 
\eqe
where $(U,\Phi)$ is an LQG surface (for some underlying Riemann surface $U$ and some central charge $\ccL$), $\alpha_1,\dots,\alpha_n \in \BB R$, $z_1,\dots,z_n \in U\cup \bdy U$, $\nu_\Phi$ denotes central charge-$\ccL$ LQG length, and $\mu > 0$ is a constant. See, e.g.,~\cite{krv-dozz,gkrv-bootstrap,ars-fzz, rz-boundary-lcft, arsz-structure-constants} for computations of such correlation functions in various settings.

The supercritical LQG disk field of Definition \ref{def-supercritical-disk} can be written as $\Phi = (Q/2) \Phi_2 + \sqrt{1-Q^2/4} \Psi$, where $ \Phi_2$ is the field for a critical LQG disk and $\Psi$ is an independent zero-boundary GFF. Since $\Psi$ is a Gaussian field with known covariance structure, it is easy to compute correlation functions of the type~\eqref{eqn-corr-function} with $\Psi$ in place of $\Phi$. Therefore, if we could compute correlation functions of $ \Phi_2$ then we could also compute correlation functions of $\Phi$. Note that our correlation functions for $\Phi$ have no bulk cosmological constant term (as in~\cite{rz-boundary-lcft}) since we do not have a finite area measure associated for supercritical LQG (c.f.~\cite{bgs-supercritical-crt}). 

The existing results on correlation functions only treat the subcritical case $\gamma < 2$. This provides additional motivation for the following problem.  

\begin{prob} \label{prob-critical-corr}
Extend the formulas for Liouville correlation functions appearing in, e.g.,~\cite{krv-dozz,gkrv-bootstrap,ars-fzz, rz-boundary-lcft,arsz-structure-constants} to the critical case $\gamma=2$. Use the $\gamma=2$ analogs of the formulas in~\cite{rz-boundary-lcft} to compute correlation functions for the supercritical LQG disks appearing in Definition \ref{def-supercritical-disk}.
\end{prob}
 
Often, when one has a one-parameter families of models which exhibit some sort of duality relation, one sees some special behavior at the self-dual point. As explained in Section~\ref{sec-duality}, for the supercritical LQG disks of central charge $\ccL\in (1,25)$ introduced in this paper, there is a duality relation between central charge $\ccL$ and central charge $26-\ccL$. 

\begin{prob} \label{prob-self-dual}
Is there anything special about LQG with the self-dual central charge value $\ccL =13$?
\end{prob} 

Currently, we do not have any guess as to what sort of special property one might have in the setting of Problem~\ref{prob-self-dual}. We note, however, that work by Gervais, et.\ al.\ (see, e.g.,~\cite{gervais-weak-to-strong,gn-locality-string-models,bg-new-critical-dim}) shows that the cases when $\ccL \in \{7,13,19\}$ are special from an algebraic perspective.

\appendix

\section{Liouville fields and rotations thereof}
\label{sec-disk-rotate}

In this appendix we introduce the Liouville field, following~\cite{ahs-integrability}, and prove that one can ``rotate'' a vector of independent Liouville fields by an orthogonal matrix to get another vector of Liouville fields with the same total central charge (Proposition~\ref{prop-disk-rotate}). This elementary fact is of independent interest and also provides additional justification for Definition~\ref{def-supercritical-disk}. An analogous statement in the setting of imaginary geometry plays a central role in~\cite{ag-mismatched}.

We will work with infinite measures on quantum disks since LQG areas and lengths are not in general preserved under taking linear combinations of fields (this only works in the setting of Section~\ref{sec-supercritical-disk} since one of the fields vanishes on the boundary). We also work on the half-plane $\BB H$ instead of the disk $\BB D$ in this subsection to be consistent with~\cite{ahs-integrability}. 
The following definition corresponds to~\cite[Definition 2.4]{ahs-integrability}, although that paper only allows $\ccL  >25$.  

\begin{defn} \label{def-liouville-field}
Let $\ccL \geq 1$ and as per usual let $Q = \sqrt{(\ccL-1)/6}$. 
Let $ \Phi^0$ be a free-boundary GFF on $\BB H$, with the additive constant chosen so that the average of $ \Phi^0$ over the unit semicircle is zero. 
Independently from $\Phi^0$, let $C \in \BB R$ be ``sampled'' from the infinite measure $e^{-Q c} \,dc$. 
The \textbf{Liouville field on $\BB H$ with central charge $\ccL$} is the generalized function
\eqb \label{eqn-liouville-field} 
\Phi(\cdot) := \Phi^0(\cdot) - 2Q \log \max\{|\cdot|,1\} + C  . 
\eqe 
We write $\op{LF}_{\BB H}(\ccL)$ for the infinite measure which is the law of $\Phi$.
That is, for each measurable subset $\mcl U$ of the space of generalized functions on $\BB H$, $\op{LF}_{\BB H}(\ccL)(\mcl U)$ is equal to the total mass of the set of pairs $(\Phi^0,C)$ for which $\Phi \in \mcl U$ under the measure $(\text{law of $\Phi^0$})\times e^{-Q c}\,dc$.
\end{defn}

For $\ccL = 1$ we have $Q = 0$, and hence the Liouville field with $\ccL = 1$ is the free-boundary GFF plus a constant sampled from Lebesgue measure on $\BB R$. 

In the subcritical case $\ccL>25$, the Liouville field is closely connected to the LQG disk. 
Indeed, let $\BB M_{\ccL}$ be the infinite measure on unmarked LQG disks of central charge $\ccL$, as defined in~\cite[Section 4.5]{wedges} (see also~\cite[Definition 2.2]{ahs-integrability}).  
The measure on LQG disks with (subcritical) LQG boundary length $L$ is obtained as the conditional law of a ``sample'' from $\BB M_{\ccL}$ given that the boundary length is $L$.
Let $(\BB H , \Phi)/{\sim}_Q$ be a sample from the infinite measure $\BB M_{\ccL}$. Independently, let $f : \BB H \to \BB H$ be a conformal automorphism of $\BB H$ sampled from the (infinite) Haar measure on the group of such conformal automorphisms (which is isomorphic to $\op{PSL}_2(\BB R)$). 
We call the field $\Phi\circ f + Q\log |f'|$ the \textbf{uniform embedding} of the LQG disk. It is shown in~\cite[Theorem 1.2]{ahs-integrability} that for $\ccL > 25$, the law of the uniform embedding of the LQG disk is a constant multiple of $\op{LF}_{\BB H}(\ccL)$. The precise constant depends on the choice of Haar measure (which is unique up to multiplicative constant).
 
Exactly analogous statements hold with the Liouville field on $\BB C$~\cite[Definition 2.24]{ahs-integrability} instead of on $\BB H$ and the uniform embedding of the LQG sphere in place of the uniform embedding of the LQG disk (this case is also treated in~\cite[Theorem 1.2]{ahs-integrability}). We refer to~\cite{dkrv-lqg-sphere} and~\cite[Section 4.5]{wedges} for two different definitions of the LQG sphere and~\cite{ahs-sphere} for a proof of their equivalence.

\begin{remark} \label{remark-free-boundary}
In the critical case $\ccL = 25$, $\op{LF}_{\BB H}(\ccL)$ should be a constant multiple of the law of the uniform embedding of the critical LQG disk, exactly as in the subcritical case, but a proof has not been written down. 
For $\ccL \in (1,25)$, it is natural to consider the infinite measure on supercritical LQG surfaces which is the law of $(\BB H , \Phi)/{\sim_Q}$, where $\Phi$ is as in~\eqref{eqn-liouville-field}.  
Note that this LQG surface behaves like a free-boundary GFF near the boundary, in contrast to the supercritical LQG disk from Definition~\ref{def-supercritical-disk}. Consequently, there is not an obvious canonical notion of ``area'' or ``boundary length'' for $\op{LF}_{\BB H}(\ccL)$ which we can condition on to get a finite measure.\footnote{\label{footnote-free-boundary} 
An additional intuitive reason for why we should use a zero-boundary GFF instead of a free-boundary GFF in Definition~\ref{def-supercritical-disk} is as follows. As explained in Section~\ref{sec-duality}, in the coupling considered in this paper the supercritical LQG disk plays the role of both the LQG field and the matter field. In couplings of critical LQG on the disk with an independent matter field, one typically takes the matter field to be a zero-boundary GFF. Hence it is reasonable to want our supercritical LQG disk to have boundary conditions which are in some sense between those of a free-boundary GFF and a zero-boundary GFF.}  
We expect that it is possible to extend the ideas of this paper to get couplings of $\op{LF}_{\BB H}(\ccL)$ for $\ccL\in(1,25)$ with SLE$_4$-type sets (e.g., the arc loop ensemble which describes the level lines of a free-boundary GFF~\cite{qw-common-level-lines}). These couplings should have the property that the LQG surfaces parametrized by the complementary connected components of the set are conditionally independent supercritical LQG disks (in the sense of Definition~\ref{def-supercritical-disk}) given their boundary lengths. 
The above discussion also applies essentially verbatim in the case of the Liouville field on the sphere. 
\end{remark}

\begin{prop} \label{prop-disk-rotate}
Let $n\in\BB N$, let $\cc_1  , \dots , \cc_n \geq 1$, and let $Q_1, ,\dots , Q_n \geq 0$ be the corresponding background charges. 
Let $(\Phi_1 , \dots ,\Phi_n)$ be sampled from the infinite measure $\op{LF}_{\BB H}(\cc_1) \times \dots \times \op{LF}_{\BB H}(\cc_n)$. Let $A$ be an $n\times n$ orthogonal matrix and define 
\eqb \label{eqn-rotated-fields}
\left( \begin{array}{c}
\wh\Phi_1 \\ 
\vdots \\ 
\wh\Phi_n
\end{array}  \right) 
:= A \left( \begin{array}{c}
 \Phi_1 \\ 
\vdots \\ 
 \Phi_n
\end{array}  \right) \quad \text{and} \quad 
\left( \begin{array}{c}
\wh Q_1 \\ 
\vdots \\ 
\wh Q_n
\end{array}  \right) 
:= A \left( \begin{array}{c}
  Q_1 \\ 
\vdots \\ 
  Q_n
\end{array}  \right) .
\eqe  
Also let $\wh\cc_i = 1 + 6 \wh Q_i^2$ for $i=1,\dots,n$. 
Then $\sum_{i=1}^n \wh \cc_i = \sum_{i=1}^n \cc_i$.
Furthermore, if $\wh Q_1,\dots,\wh Q_n \geq 0$, then the law of $(\wh\Phi_1,\dots,\wh\Phi_n) $ is $\op{LF}_{\BB H}(\wh\cc_1) \times \dots \times \op{LF}_{\BB H}(\wh\cc_n)$.
Exactly the same statement is true if we replace the law $\op{LF}_{\BB H}(\cc)$ of the Liouville field on $\BB H$ by the law of the Liouville field on $\BB C$~\cite[Definition 2.24]{ahs-integrability}.
\end{prop}
\begin{proof}
Since $A$ is an orthogonal matrix, it preserves Euclidean norms and hence
\eqbn
\sum_{i=1}^n \wh \cc_i =n + 6 \sum_{i=1}^n \wh Q_i^2 = n + 6 \sum_{i=1}^n Q_i^2 = \sum_{i=1}^n \cc_i .
\eqen
 
For $i = 1,\dots,n$, write 
\eqb \label{eqn-liouville-field-i}
\Phi_i(\cdot) := \Phi^0_i(\cdot) - 2Q_i \log \max\{|\cdot|,1\} + C_i  
\eqe 
as in~\eqref{eqn-liouville-field}, where $\Phi^0_1,\dots,\Phi^0_n$ are independent free-boundary GFFs on $\BB H$ and the law of $(C_1,\dots,C_n) $ is $\exp\left( - \sum_{i=1}^n Q_i c_i \right) \,dc_1\, \dots ,\, dc_n$. We can take $(\Phi^0_1,\dots,\Phi^0_n)$ and $(C_1,\dots,C_n)$ to be independent (i.e., their joint law factors as the product of their marginal laws). 

Define
\eqbn
\left( \begin{array}{c}
\wh\Phi^0_1 \\ 
\vdots \\ 
\wh\Phi^0_n
\end{array}  \right) 
:= A \left( \begin{array}{c}
\Phi^0_1 \\ 
\vdots \\ 
\Phi^0_n
\end{array}  \right)  \quad\text{and} \quad
\left( \begin{array}{c}
\wh C_1 \\ 
\vdots \\ 
\wh C_n
\end{array}  \right) 
:= A \left( \begin{array}{c}
C_1 \\ 
\vdots \\ 
C_n
\end{array}  \right)  .
\eqen
Since $\Phi^0_1,\dots,\Phi^0_n$ are i.i.d.\ Gaussian processes, the rotational invariance of the $n$-dimensional Gaussian distribution shows that $\wh\Phi^0_1,\dots,\wh\Phi^0_n$ are independent free-boundary GFFs on $\BB H$. By making a linear change of variables, we see that $(\wh C_1,\dots,\wh C_n)$ has the law $\exp\left( - \sum_{i=1}^n \wh Q_i \wh c_i \right) \,d\wh c_1\, \dots ,\, d\wh c_n$ and is independent from $(\wh\Phi^0_1,\dots,\wh\Phi^0_n)$. 
Plugging~\eqref{eqn-liouville-field-i} into the definition~\eqref{eqn-rotated-fields} of $\wh \Phi_1,\dots,\wh\Phi_n$ now concludes the proof. The statement for the LQG sphere is proven in the same way. 
\end{proof}

\begin{remark}
The linear combination in Proposition~\ref{prop-disk-rotate} behaves nicely under LQG coordinate change. That is, if $f : U \to \BB D$ is a conformal map, then it is immediate from~\eqref{eqn-rotated-fields} that
\eqbn
\left( \begin{array}{c}
\wh\Phi_1 \circ f + \wh Q_1 \log |f'| \\ 
\vdots \\ 
\wh\Phi_n \circ f + \wh Q_n \log |f'|
\end{array}  \right) 
= A \left( \begin{array}{c}
 \Phi_1 \circ f + Q_1 \log |f'| \\ 
\vdots \\ 
 \Phi_n \circ f + Q_n \log |f'| 
\end{array}  \right) .
\eqen
\end{remark}

\bibliography{paper-references}
\bibliographystyle{hmralphaabbrv}

\end{document}